\documentclass[reqno,centertags]{amsart}
\usepackage{amsmath,amsthm,amscd,amssymb}
\usepackage{latexsym}
\usepackage{bbm}
\usepackage{color}

\newcommand{\ii}{{\mathrm{i}}}
\newcommand{\R}{\mathbb{R}}
\newcommand{\N}{\mathbb{N}}
\newcommand{\T}{\mathbb{T}}

\newcommand{\bdone}{{\boldsymbol{1}}}

\newcommand{\lb}{\label}

\newcommand{\ol}{\overline}

\newcommand{\tr}{\text{\rm{Tr}}}

\newcommand{\coeff}{\text{\rm{coeff}}}
\newcommand{\meas}{\text{\rm{meas}}}

\newcommand{\beq}{\begin{equation}}
\newcommand{\eeq}{\end{equation}}
\newcommand{\ba}{\begin{align}}
\newcommand{\ea}{\end{align}}

\newcommand{\SC}{\operatorname{SC}}
\newcommand{\MP}{\operatorname{MP}}

\DeclareMathOperator{\GW}{GW}

\def \ud{{\tt u}}

\def \w{{\tt w}}
\def \g{{\tt g}}
\def \a{{\mathfrak a}}

\def \sur#1#2{\mathrel{\mathop{\kern 0pt#1}\limits^{#2}}}

\def \z{{\mathfrak z}}
\newcommand{\al}{\mathfrak \alpha}
\newcommand{\sn}{^{(n)}}
\newcommand{\Sr}{\mathcal{S}}
\newcommand{\ap}{\alpha^+}
\newcommand{\am}{\alpha^-}

\newcommand{\abold}{\boldsymbol{\alpha}}
\newcommand{\muun}{\mu^{(n)}_\ud}
\newcommand{\mbold}{\boldsymbol{m}}
\newcommand{\mubold}{\boldsymbol{\mu}}

\newcommand{\Tbold}{\boldsymbol{\Theta}}

\newcommand{\Rbold}{\boldsymbol{R}}
\newcommand{\hbold}{\boldsymbol{h}}
\definecolor{Red}{rgb}{1,0,0}
\definecolor{Blue}{rgb}{0,0,1}

\newcounter{smalllist}

\newcommand{\CUE}{\operatorname{CUE}}
\newcommand{\GUE}{\operatorname{GUE}}

\newcommand{\LUE}{\operatorname{LUE}}

\DeclareMathOperator\sign{sgn}
\DeclareMathOperator\jac{Jac}
\DeclareMathOperator\diag{Diag}
\allowdisplaybreaks
\numberwithin{equation}{section}

\newtheorem{theorem}{Theorem}[section]

\newtheorem{lemma}[theorem]{Lemma}
\newtheorem{corollary}[theorem]{Corollary}
\theoremstyle{definition}

\newtheorem*{acknowledgement}{Acknowledgement}

\theoremstyle{remark}
\newtheorem{remark}{Remark}
\newtheorem*{remarks}{Remarks}

\newcommand{\abs}[1]{\lvert#1\rvert}

\usepackage[pdftex,colorlinks=true]{hyperref}

\begin{document}
\title[Spectral measures of  spiked matrices]
{Large deviations for spectral measures of some spiked matrices}
\author[N.~Noiry and A.~Rouault]
{Nathan Noiry$^{1}$ and
Alain Rouault$^2$}

\thanks{$^1$ Telecom Paris, 91120 Palaiseau France, e-mail: noirynathan@gmail.com}

\thanks{$^2$ Laboratoire de Math{\' e}matiques de Versailles, UVSQ, CNRS, Universit\'e Paris-Saclay, 78035-Versailles Cedex France, e-mail: alain.rouault@uvsq.fr}

\date{\today}

\begin{abstract} 
We prove large deviations principles for spectral measures of perturbed (or spiked) matrix models in the direction of an eigenvector of the perturbation. In each model under study, we provide two approaches, one of which relying on large deviations principle of unperturbed models derived in the previous work "Sum rules via large deviations" (Gamboa-Nagel-Rouault, {\it JFA} \cite{GaNaRo} 2016). 
\end{abstract}

\keywords{Large deviations, Sum rules, Jacobi coefficients, Verblunsky coefficients,  Matrix measures,  Relative entropy}
\subjclass[2010]{60F10, 60G57, 60B20, 47B36}
\maketitle

\section{Introduction} \lb{s1}

Beside the empirical spectral distribution of a $n\times n$ random matrix $M_n$ 
\[\mu_\ud\sn = \frac{1}{n} \sum_{k=1}^n \delta_{ \lambda_k}, \]
whose asymptotical behavior is widely known for numerous matrix models, there has been a growing interest in the study of the so-called spectral measures. For any fixed unit vector $e\sn \in \mathbb C^n$, the spectral measure associated to the pair $(M_n,e\sn)$ is
the probability measure $\mu_\w\sn$ defined by
\begin{align*}
\langle e\sn , (M_n -z)^{-1} e\sn\rangle = \int_{\mathbb R} \frac{d\mu_\w\sn (x)}{x-z} \ \ \hbox{for all} \  z \in \mathbb C \setminus \mathbb R
\end{align*}
if $M_n$ is Hermitian or  
\begin{align*} 
\large\langle e\sn , \frac{M_n + z}{M_n -z} e\sn\large\rangle = \int_0^{2\pi} \frac{e^{\ii \theta} +z}{e^{\ii \theta} -z}  d\mu_\w\sn (\theta) \ \  \hbox{for all} \ z : |z|\not= 1,
\end{align*}
if $M_n$ is unitary. 
In turns out that the spectral measure is a weighted version of the empirical spectral distribution:
\[\mu_\w\sn =  \sum_{k=1}^n \w_k \delta_{ \lambda_k}, \]
where $\w_k = |\langle \phi_k , e\sn\rangle|^2$, with $\phi_k$ a unit eigenvector associated to the eigenvalue $\lambda_k$. It was studied under the name {\it eigenvector empirical spectral distribution} in \cite{XYY}, in the context of unperturbed random covariance matrices.

In a series of papers \cite{gamboa2011large,GaNaRo,GNROPUC,GaNaRomat,GNRis} Gamboa et al.  studied the random spectral measure $\mu_\w\sn$ of a pair $(M_n, e\sn)$ where $M_n$ is a random $n \times n$ matrix self-adjoint or unitary, whose distribution is invariant by conjugation, and $e\sn$ is a fixed vector of $\mathbb C^n$. When the Radon-Nikodym density of this distribution is of the form $\exp -n^2 \tr\!\ V(M_n)$ and 
 with convenient assumptions on the potential $V$, the authors proved that the family $(\mu_\w\sn)_{n \geq 1}$ satisfies a large deviations principle at scale $n$ with a good rate function consisting of two parts. The first part is the  Kullback entropy of the equilibrium measure $\mu_V$ with respect to the absolute continuous part of the argument measure. The second part corresponds to the contribution of the outliers of the argument measure, namely of the eigenvalues that belong to the complement of the support of $\mu_V$. Besides, when the spectral measure is encoded by the Jacobi recursion coefficients (or the Verblunsky  coefficients in the unitary case), the rate function admits another expression in term of these coefficients, which is a simple functional in most of the classical cases. The  identification of the two expressions of the rate functions leads to the so called {\it sum rules}.
 
The simplest Hermitian invariant models are the well known Gaussian Unitary Ensemble $\GUE (n)$ and Laguerre Unitary Ensemble $\LUE_{n\tau}(n)$, whose equilibrium measures are respectively given by the semi-circle law ($\SC$) and the Marchenko-Pastur law ($\MP_\tau$). In the unitary world, the simplest model is of course the $\CUE(n)$ which corresponds to the Haar measure on the unitary group. The first non-trivial models are provided by the Gross-Witten measures 
$\mathbb G\mathbb W_\g(n)$ which form a family of probability measures on the unitary group, absolutely continuous with respect to the $\CUE(n)$, parametrized by a real number $\g$.

In this paper, we are interested in the large deviations of the spectral measures of rank-one perturbations of the classical aforementioned models. More precisely, we will consider additive perturbations of the $\GUE (n)$, multiplicative perturbations of the $\LUE_{n\tau}(n)$ and multiplicative perturbation of the Gross-Witten measures.  

The first model of spiked random matrices was proposed by Johnstone \cite{Johnstone}, who was motivated by several statistical reasons. Among others, the largest eigenvalues (and their associated eigenvectors) of the variance-covariance matrix of some data points is at the basis of the so-called Principal Component Analysis. With the current ability to collect and store massive databases, the practitioner is often faced with a number of observations ($n$) of the same order as their dimension ($p$), which makes the study of large random matrices relevant, at least to understand the mechanisms underlying the behavior of the spectrum. This initial observation of Johnstone has led Baik, Ben Arous and P\'ech\'e to find their famous phase transition \cite{BBPphase}. Since then, a tremendous amount of work has been conducted on spiked models. we refer the reader to \cite{DMC} for a survey of the afferent literature.

Let us mention that, at the level of large deviations, the extreme eigenvalues have been studied in \cite{BGGM}, and the pair (extreme eigenvalue, weight) has been recently considered in \cite{biroli}. In the present work, we establish large deviations principles for the sequences of spectral measures associated to the pairs $(M_n,e^{(n)})$, in case where the reference vector $e^{(n)}$ is colinear to the eigenvector of the perturbation. The corresponding good rate functions are simple perturbations of the good rate functions of the undeformed models and we refer the reader to Theorems \ref{theo:perturbGUE}, \ref{corLUE} and \ref{corGW} for precise statements. 

In order to derive these large deviations principles, we propose two approaches, each based on the already known LDP for classical models, and shedding different lights on the problem. The first one uses that the distributions of the spectral measures of the deformed models are tilted versions of the distributions of the spectral measures of the undeformed ones. The second approach relies on the computations of the Jacobi (resp. Verblunsky) parameters of the deformed models.

Of course, the unique minimizers of the rate functions corresponds to the limiting spectral measures of the considered models. In particular, we recover the expressions of the limiting spectral measures associated to the perturbations of the $\GUE(n)$ and the $\LUE_{n \tau}(n)$, which belong to the class of {\it free Meixner laws}. In the Gaussian setting, this was first observed in \cite{Lenczewski}. In the general case, this is a consequence of the local laws \cite{lee2015edge,knowles-yin}, as observed in \cite{noiry2019spectral}. For related papers on finite rank perturbations, see \cite{Kozhan} and \cite{Webb}; on Meixner class see \cite{Bryc}.

A byproduct of our considerations also yields a characterization of the limiting measures as the unique minimizers of the rate functions of the unperturbed models, under a constraint on the mean.

In a last part, we propose two generalizations.
The first one is concerned with perturbations of general invariant models, while the second one deals with matricial versions of the spectral measures.

In all the sum rules considered, the Kullback-Leibler divergence or relative entropy between two probability measures  
$\mu$ and $\nu$ plays a major role. When the probability space is $\mathbb R$ endowed with its Borel  $\sigma$-field,  
it  is defined by
\begin{equation}
\label{KL}
{\mathcal K}(\mu\, |\ \nu)= \begin{cases}  \ \displaystyle\int_{\mathbb R}\log\frac{d\mu}{d\nu}\!\ d\mu\;\;& \mbox{if}\ \mu\ \hbox{is absolutely continuous with respect to}\ \nu ,\\
   \  \infty  &  \mbox{otherwise.}
\end{cases}
\end{equation}
Usually, $\nu$ is the reference measure. Here  the spectral side  will involve the reversed Kullback-Leibler divergence, where $\mu$ is the reference measure and $\nu$ is the argument. 

The outline of the paper is as follows. In Section \ref{sec:notation}, we present our three random models and the main notations. Section \ref{sec:recapOPRLOPUC} gives the encoding of the spectral measures by Jacobi parameters in the real case and Verblunsky parameters in the complex case. In Section \ref{sec:recapsumrules}, we recall the results obtained by the second author of this paper with Gamboa and Nagel about large deviations and 
sum rules. Section \ref{sec:LDPperturb} contains our results, which are stated in Theorems \ref{theo:perturbGUE}, \ref{corLUE} and \ref{corGW}. In Section \ref{sec:generalization}, we present some generalizations in Theorems \ref{6.3}, \ref{6.4} and \ref{6.5}. 
 Finally, in an appendix we present a technical lemma and a short panorama of measures found in the different limits, which simplifies some computations along the paper.

\section{Notations} \label{sec:notation}

In this article, we are going to consider perturbed versions of three classical models of random matrices whose definitions are recalled here. The two first models have real eigenvalues and correspond to the Hermite and the Laguerre ensembles. The third model will have its eigenvalues on $\mathbb{T} :=\{ z \in \mathbb{C}, \, |z|=1\}$, and corresponds to the so-called Gross-Witten measure, which is absolutely continuous with respect to the Haar measure on $\mathbb U (n)$.  We denote by $\mathcal M_1(\mathbb R)$ (resp. $\mathcal M_1(\mathbb T)$) the set of probabilty measures on $\mathbb R$ (resp. $\mathbb T$).
\medskip
\paragraph*{{\bf The Hermite ensemble.}}

For all $n \geq 1$, the Gaussian Unitary Ensemble $\GUE(n)$, or Hermite ensemble, is a probability distribution on Hermitian matrices of size $n \times n$, whose density is proportional to  $\exp \left( -\frac{1}{2} \tr (X_nX_n^\star) \right)$ with respect to the Lebesgue measure $dX_n$. The rescaled matrix $H_n = \frac{1}{\sqrt{n}}X_n$ has law:
\begin{align}
\lb{defGUE}
\mathbb P_0\sn (dH) := \frac{1}{\mathcal Z_n} \exp \left( -\frac{n}{2} \tr (HH^\star) \right)  dH\,,
\end{align}
where $\mathcal Z_n$ is the normalization constant.\footnote{All the normalization constants will be denoted by the same symbol, without possible confusion since the different models are treated separately.}

The equilibrium measure of this ensemble, i.e. the limit of the empirical spectral distribution $\mu_\ud\sn$ is the semicircle distribution :
\begin{equation}
\label{SC0}
\operatorname{SC}(dx) = \frac{1}{2\pi}\sqrt{4-x^2}\!\ \mathbbm{1}_{[-2, 2]}(x)\!\ dx.
\end{equation} 
\medskip
\paragraph*{{\bf The Laguerre ensemble.}}
For all $n \geq 1$, let $N=N(n)$ be such that $n \leq N$. Let $X_n$ be a $n \times N$ complex matrix with i.i.d. Gaussian entries  whose real and imaginary parts are i.i.d. $\mathcal N(0; 1/2)$.
 Then, the Laguerre Unitary Ensemble $\mathrm{LUE}_N(n)$ is the distribution of $X_nX_n^\star$, whose density  is proportional to $(\det XX^\star)^{N-n} \exp \left( -  \tr\!\ XX^\star \right)$. The law of the rescaled matrix $L_n = \frac{1}{N} X_nX_n^\star$ is therefore given by
\begin{align}
\lb{defLUE}
\mathbb Q_1\sn (dL) := \frac{1}{\mathcal Z_n} 
(\det L)^{N-n} \exp \left( - N  \tr\!\ L \right) dL\,. 
\end{align}
All along this article, we will assume that  $N/n \rightarrow \tau^{-1} > 1$ as $n \rightarrow +\infty$. The equilibrium measure of this Laguerre ensemble, 
 is the  Marchenko-Pastur distribution with parameter $\tau$:
\begin{align}
\MP_\tau(dx) = \frac{\sqrt{(\tau^+ -x)(x-\tau^-)}}
{2\pi \tau x} \ \mathbbm{1}_{(\tau^-, \tau^+)} (x) dx
\label{MP0} 
\end{align}
where $\tau^\pm := (1\pm \sqrt\tau)^2$.
\medskip
\paragraph{{\bf The Gross-Witten ensemble.}}
Our third model has its eigenvalues on $\mathbb T$ and corresponds to the Gross-Witten measure $\mathbb G\mathbb W_{\g}(n)$ with parameter $\g \in \mathbb R$. It is a probability measure on the unitary group $\mathbb{U}(n)$ given by
\begin{equation}
\label{GW0}
\mathbb R_0\sn(dU) =  \frac{1}{\mathcal Z_n} \exp \left[\frac{n\g}{2}\tr\!\ (U + U^\star)\right] dU
\end{equation}
where $dU$ is the Haar probability measure on $\mathbb U(n)$.
Let us mention that the Gross-Witten measure arises in the context of the Ulam's problem which concerns the length of the longest increasing subsequence inside a uniform permutation \cite{DJ}.
For other details and applications of this distribution we refer to \cite{HiaiP} p. 203, \cite{GrossW}, \cite{Wadia}. 

There are two different behaviors according to the value of the parameter $\g$.

For  $|\g| \leq 1$ (ungapped or strongly coupled phase). In this context, the equilibrium measure $\GW_\g$ is supported on $\mathbb T$ and has the following density:
\begin{equation}
\label{GW-}
\GW_\g(dz)
= \frac{1}{2\pi} (1 + \g \cos \theta)\!\ d\theta,\; (z = e^{\ii \theta } , \theta \in [-\pi, \pi)).
\end{equation}
Note that $ \GW_\g$ has only nontrivial moments of order  $\pm 1$.

For $|\g| > 1$, the equilibrium measure is supported by an arc. This case will  not be considered here since the paper would be lenghtened with involved computations.

\section{Recap on Orthogonal polynomials} \label{sec:recapOPRLOPUC}
In this section we recall the possible parametrization of positive measures on $\mathbb{R}$ (resp. $\mathbb T$) by their Jacobi (resp. Verblunsky) coefficients. The latter appear through the spectral theory of orthogonal polynomials on the real line (OPRL), resp. the spectral theory of orthogonal polynomials on the unit circle (OPUC), which we briefly recall here. In the next section, we will use these parametrizations in order to recall the large deviations principles satisfied by the spectral measures of the models defined in Section \ref{sec:notation}.

\subsection{OPRL}
Let $\rho$ be a positive measure on $\mathbb R$ whose support is bounded but not made of a finite union of points. Let $(p_n(x))_{n \geq 0}$ be the sequence of orthonormal polynomials associated to $\rho$, obtained by applying the Gram-Schmidt algorithm to the basis $\{1,x,x^2, \ldots\}$. Then, there exists two sequences of uniformly bounded real numbers $(a_n)_{n \geq 0}$ and $(b_n)_{n \geq 0}$ such that $a_n>0$ for all $n \geq 0$ and such that the polynomials $p_n(x)$'s satisfy the following three terms recursion:
\begin{equation} \lb{1.1}
xp_n(x) =a_{n+1} p_{n+1}(x) + b_{n+1} p_n(x) + a_n p_{n-1}(x).
\end{equation}
The parameters $\{a_n,b_n\}_{n=1}^\infty$ are called the {\it Jacobi parameters} associated to $\rho$. We will denote
\begin{align}
\label{Jacarray}
\jac(\rho) = \begin{pmatrix} b_1,&b_2,&\cdots\\
a_1,&a_2,&\cdots
\end{pmatrix}\,.
\end{align}
As it is well known (see, e.g., \cite[Section~1.3]{simon05}), Equation \eqref{1.1} sets up the one-to-one correspondence between uniformly bounded sequences $(a_n)_{n \geq 1}$, $(b_n)_{n \geq 1}$ and positive measures $\rho$ on $\mathbb R$ whose supports are bounded but not made of a finite union of points. Moreover, a similar argument implies that there exists a one-to-one correspondence between the set of positive measures $\rho$ on $\mathbb R$ whose support are finite union of $N$ distinct points and the set of sequences $(a_n)_{1 \leq n \leq N-1}$ and $(b_n)_{1 \leq n \leq N}$ such that $a_n > 0$ for all $1 \leq n \leq N$. Let us mention that the Jacobi parameters of the semicircle law are given by:
\begin{align}
\label{JacSC}
\jac(\SC) = \begin{pmatrix} 0,&0,&\cdots\\
1,&1,&\cdots
\end{pmatrix}\,,
\end{align}
it is called the ``free" case in the OPRL literature.

When $\rho$ is supported on $[0,\infty)$ the recursion coefficients can be decomposed as
\begin{align} \label{zerl1}
\begin{split}
b_k =& z_{2k-2} + z_{2k-1}, \\
a_k^2 =&  z_{2k-1}z_{2k},
\end{split}
\end{align}
for $k \geq 1$, where $z_k\geq 0$ and $z_{0}=0$. In fact, by Favard's Theorem a measure $\rho$ is supported on $[0,\infty)$ if and only if its Jacobi coefficients satisfy the decomposition \eqref{zerl1}. In particular, the $\MP_\tau$ distribution corresponds to
$z^{\MP}_{2n-1}=1$ and 
$z^{\MP}_{2n} = \tau$ for all $n \geq 1$, so that
\begin{align}
\label{JacMP}
\jac(\MP_\tau) = \begin{pmatrix} 1,&1+\tau,&1+\tau,& \cdots\\
\sqrt\tau,&\sqrt\tau,&\sqrt\tau,&  \cdots
\end{pmatrix}\,.
\end{align}

Let us finally mention that the measure $\rho$ can be realized as the spectral measure associated to the pair $(J,e_1)$, where $J$ is the so-called Jacobi matrix which represents the multiplication by $x$ in the basis $(p_n(x))_{n \geq 0}$ of $L^2(\rho)$:
\begin{equation} \lb{1.3a}
J := \begin{pmatrix}
b_1 & a_1 & 0 & \vphantom{\ddots} \\
a_1 & b_2 & a_2 & \ddots  \\
0 & a_2 & \ddots & \ddots   \\
{} & \ddots & \ddots & \ddots
\end{pmatrix}.
\end{equation}

\subsection{OPUC}
Let $\mu$ be a probability measure on $\mathbb T$ whose support is not a finite set of points. Let $(\varphi_n(z))_{n \geq 0}$ be the sequence of orthonormal polynomials associated to $\mu$, obtained by applying the Gram-Schmidt algorithm to the basis $\{1,z,z^2, \ldots\}$. Then, there exists a sequence of complex numbers $(\alpha_n)_{n \geq 0}$, called the {\it Verblunsky coefficients} associated to $\mu$, such that $|\alpha_n| <1$ for all $n \geq 0$ and such that the polynomials $\varphi_n(z)$'s satisfy the following recursion:
\begin{equation} \lb{1.2}
z\varphi_n(z) = \rho_n \varphi_{n+1}(z) + \bar\alpha_n \varphi_n^*(z),
\end{equation}
where
\begin{equation} \lb{1.3}
\varphi_n^*(z) = z^n \, \ol{\varphi_n (1/\bar z)} \qquad
\rho_n = (1-\abs{\alpha_n}^2)^{1/2}\,.
\end{equation}
Equation \eqref{1.2} sets up a one-to-one correspondence between sequences $(\alpha_n)_{n \geq 0}$ with values inside $\{ z, \, |z| < 1\}$ and the set of positive measures $\mu$ on $\mathbb{T}$ whose supports are not finite union of points. Moreover, a similar argument implies that there exists a one-to-one correspondence between the set of positive measures $\mu$ on $\mathbb T$ whose support are finite union of $N$ distinct points and the set of sequences $(\alpha_n)_{1 \leq n \leq N}$ such that $|\alpha_n| < 1$ for all $0 \leq n \leq N-1$ and $|\alpha_N|=1$.

The sequence $\alpha_n\equiv 0$ corresponds to $\lambda_0$, the normalized Lebesgue measure on $\mathbb T$, and is called ``free" case in the OPUC literature.

For the Gross-Witten model, when $|\g|\leq 1$, 
the V-coefficients 
are given by (see \cite{simon05}, p. 86):
\begin{equation}
\label{alphalimGW}
\alpha_n^{\GW} = \begin{cases}\displaystyle  -\frac{x_+ - x_-}{x_+^{n+2} - x_-^{n+2}} & \hbox{if} \ |\g| < 1\\
\displaystyle \frac{(-\g)^{n+1}}{n+2}& \hbox{if} \ |\g| = 1\,,
\end{cases}
\end{equation}
where $x_\pm$  
 are roots of the equation 
\[x + \frac{1}{x} = - \frac{2}{\g}\,.\]
In particular
\begin{align}
\lb{alpha0}
\alpha_0^{\GW} = \frac{\g}{2}\,.
\end{align}

\section{Recap on LDP and sum rules}\label{sec:recapsumrules}

See  \cite{demboz98, deuschel2001large, agz} for background on LDP.  For the self-adjoint models the sequence $(\muun)$ satisfies the LDP at scale $n^2$ with good rate function involving the logarithmic entropy and the potential. Moreover the extremal eigenvalues satisfy the LDP at scale $n$ with a rate function $\mathcal F_H$ and $\mathcal F_L^\pm$, which represent effective potentials.  For the unitary model studied here, the support of the limiting measure is the whole unit circle and there exists no outlier. The following results are the sum rules obtained by the second author of this paper together with Gamboa and Nagel.
\newpage

\subsection{OPRL}
\subsubsection{LDP on the measure side}
To begin with, let us give some notations.
Let $\Sr = \Sr(\am,\ap)$ be the set of all bounded positive measures $\mu$ on $\R$ with 
\begin{itemize}
\item[(i)] $\operatorname{supp}(\mu) = J \cup \{E_i^-\}_{i=1}^{N^-} \cup \{E_i^+\}_{i=1}^{N^+}$, where $J\subset I= [\am,\ap]$, $N^-,N^+\in\N\cup\{\infty\}$ and 
\begin{align*}
E_1^-<E_2^-<\dots <\am \quad \text{and} \quad E_1^+>E_2^+>\dots >\ap .
\end{align*}
\item[(ii)] If $N^-$ (resp. $N^+$) is infinite, then $E_j^-$ converges towards $\am$ (resp. $E_j^+$ converges to $\ap$).
\end{itemize}
Such a measure $\mu$ will be written as
\begin{align}\label{muinS0}
\mu = \mu_{|I} +  \sum_{i=1}^{N^+} \gamma_i^+ \delta_{E_i^+} + \sum_{i=1}^{N^-} \gamma_i^- \delta_{E_i^-},.\
\end{align}
Further, we define $\Sr_1=\Sr_1(\am,\ap):=\{\mu \in \Sr |\, \mu(\R)=1\}$. We endow $\Sr_1$ with the weak topology and the corresponding Borel $\sigma$-algebra.

On the measure side we have
\begin{theorem}
The family of distributions of $(\mu_\w\sn)$ under  $\GUE (n)$ (resp. $\LUE_N(n)$)
 satisfies the LDP on $\mathcal M_1(\mathbb R)$ equipped with the weak topology in the scale $n$ with good rate function $\mathcal I^H_\meas$ (resp. $\mathcal I^L_\meas$) given by
\begin{equation}\label{eq:defnImeasGUE}
\mathcal I^H_\meas(\mu) = 
\begin{cases}
\mathcal K(\SC\, |\ \mu) + \sum_k \mathcal F_H(E_k^\pm) \ \ \hbox{if} \  \mu \in \Sr_1(-2,2), \\
\infty \ \ \hbox{otherwise},
\end{cases}
\end{equation}
where
\begin{equation}
\mathcal{F}_H(x) :=  \begin{cases} &  \displaystyle \int_2^{|x|} \sqrt{t^2-4}\!\ dt 
\;\;\;\;\mbox{if} \ |x| \geq 2\\
    &  \infty\;\;\mbox{ otherwise,}
\end{cases}
\end{equation}
resp.
\begin{equation}
\mathcal I^L_\meas(\mu) =
\begin{cases} 
\mathcal K(\MP_\tau\, |\ \mu) + \sum_k \mathcal F
^\pm_L(E_k^\pm)\ \ \ \hbox{if} \  \mu \in \Sr_1(\tau^-, \tau^+)\\
\infty \ \ \hbox{otherwise},
\end{cases}
\end{equation}
where
\begin{align}
\mathcal{F}_L^+(x) &=  \displaystyle\int_{\tau^+}^x \frac{\sqrt{(t -  \tau^-)(t - \tau^+)}}{t\tau}\!\ dt\;\;\;\;\mbox{if} \ x \geq \tau^+,\\
{\mathcal F}_L^-(x) &= \displaystyle\int_x^{\tau^-} \frac{\sqrt{(\tau^- -t)(\tau^+ -t)}}{t\tau}\!\ dt  \;\;\;\;\mbox{if} \ x \leq \tau^-,\\
\mathcal F^\pm_L (x) &= \infty \;\;\;\;\mbox{if} \ x \in  [\tau^-, \tau^+]\,.
\end{align}
The measure $\SC$ (resp. $\MP_\tau$) is the unique minimum of $\mathcal I^H_\meas$ (resp. $\mathcal I^L_\meas$).
\end{theorem} 

Actually, we can strenghten the topology on the set of measures. We refer to  \cite{gamboa2011large} for more details.
\begin{corollary}
\label{cor42}
The above LDPs are in force  in the set $\mathcal M_1^d$ of probability measures with finite moments and determined by these moments, equipped with the topology $\mathcal T^m$ of convergence of moments.
\end{corollary}
\label{strongtopo}
\begin{proof}
For every $n$ we have $\mu_\w\sn \in \mathcal M_1^d$. Moreover the domain of the rate function is a subset of  $\mathcal M_1^d$. So by Lemma 4.1.5 (b) in \cite{demboz98}, the LDP is verified in  $\mathcal M_1^d$ with the weak convergence.
Now let, for every $M > 0$,
\[\mathcal K_M = \{\mu \in \mathcal M_1 : \mu([-M, M]^c = 0\}\,.\]
This set is compact for the  $\mathcal T^m$ topology.  
Moreover
\[\mathbb P( \mu_\w\sn \in \mathcal K_M^c) \leq \mathbb P( |\lambda\sn_{\max}| >M)\]
where $|\lambda\sn|_{\max} = \max |\lambda_k|, k = 1, \dots, n$. But this probability is bounded by $\exp -n h(M)$ with $h(M) \rightarrow \infty$ as $M\rightarrow\infty$. So, the family of distributions of $\mu_\w\sn$ is exponentially tight in $\mathcal M_1^d$ for the  $\mathcal T^m$ topology. By application of Corollary 4.2.6 in \cite{demboz98} we get the LDP for the  $\mathcal T^m$ topology.
\end{proof}

\subsubsection{Coefficients side - Sum rules}

We start by stating the classical Killip-Simon sum rule (due to  \cite{killip2003sum} and explained in \cite{Simon-newbook} p.37). 
It gives two different expressions for the discrepancy between a measure and  to the semicircle law $\operatorname{SC}$.

For a probability measure $\mu$ on $\R$ with recursion coefficients $\a := (a_k)_k,\, \mathfrak{b} :=(b_k)_k$, define 
\begin{align}\label{rateG1}
\mathcal{I}_{\coeff}^H(\a, \mathfrak{b}) := \sum_{k\geq 1}\left( \frac{1}{2} b_k^2 + G(a_k^2)\right)
\,,
\end{align}
where $G(x) = x-1 - \log x$.
It is a convex function of $(\a, \mathfrak{b})$ with values in $[0, \infty]$ which has a unique minimum at $a_k \equiv 1, b_k \equiv 0$, corresponding to the semicircle law $\SC$  (see \eqref{JacSC}).

If the support of $\mu$ is a subset of $[0, \infty)$ with $\z= (z_k)_k$, define 
\begin{align} \label{rateL1}
\mathcal{I}_\coeff^L(\z) :=   \sum_{k=1}^\infty\left( \tau^{-1}G(z_{2k-1}) +  G(\tau^{-1}z_{2k})\right)\,. 
\end{align}
It is a convex function of $\z$ with values in $[0, \infty]$ which has a unique minimum at $\z =(z_k)_{k \geq 0}$ 
with
\[z_0 = 0, z_{2k-1} = 1, z_{2k} = \tau \  (k \geq 1)\,.\]
which corresponds with $\MP_\tau$.

Then we have the following theorem.

\medskip

\begin{theorem}
\label{sumruleg}
\begin{enumerate}
\item \cite{killip2003sum}
Let $J$ be a Jacobi matrix with diagonal entries $b_k \in \mathbb R$ and subdiagonal entries $a_k > 0$ satisfying  $\sup_k a_k + \sup_k |b_k| < \infty$ and let $\mu$ be the associated spectral measure. Then $\mathcal{I}_{meas}^H(\mu)= \infty$ if $\mu \notin \Sr_1(-2,2)$ and for $\mu \in \Sr_1(-2,2)$, 
\begin{align}
\label{srH}
\mathcal I_\coeff^H(\a, \mathfrak{b}) = \mathcal I_\meas^H(\mu)
\end{align}
 where in  (\ref{srH}),  both sides may be infinite simultaneously. 
\item \cite{GaNaRo}
Assume the entries of the Jacobi matrix $J$ satisfy the decomposition 
\eqref{zerl1} with $\sup_k z_k <\infty$ and let $\mu$ be the spectral measure of $J$. Then for all $\tau \in (0,1]$, $\mathcal{I}_L(\mu)=\infty$ if $\mu \notin \Sr_1(\tau^-,\tau^+)$ and for  $\mu \in \Sr_1(\tau^-,\tau^+)$,
\begin{align}
\label{srL} 
\mathcal{I}_\coeff^L(\z) = \mathcal I_\meas^L(\mu) 
\end{align}
 where in (\ref{srL}), both sides may be infinite simultaneously. 
\end{enumerate}
\end{theorem}

Note that if $\tau=1$, the support of the limit measure is $[0,4]$, so that we have a hard edge at 0 with $N^-=0$ and no contribution of outliers to the left.

The results (1) and (2) are obtained by probablistic method (\cite{GaNaRo}) and  up to now it is the only method to prove (2).

\subsection{OPUC}
For the unitary case we have LDPs on the measure side and some sum rules. In the following $\mathcal K (\nu\, |\ \mu)$ denotes the Kullback-Leibler divergence or relative entropy of $\nu$ with respect to $\mu$ on $\mathbb T$. Here, since $\mathbb T$ is a compact set, it is enough to equip  $\mathcal M_1 (\mathbb T)$  with the topology on $\mathcal M_1 (\mathbb T)$.

\subsubsection{Measure side}

\begin{theorem}[\cite{GNROPUC} Cor. 4.5] 
\label{appliGW} 
When  $|\g| \leq 1$, the family of distributions of  $(\mu_\w\sn)$ under $\mathbb G\mathbb W_\g(n)$
 satisfies the LDP  in $\mathcal M_1(\mathbb T)$ with speed $n$ and 
 rate function 
\[\mathcal I^{\GW}_{meas} (\mu) = \mathcal K(\GW_\g\, |\ \mu)\,.\] 
The measure $\GW_\g$ is the unique minimum of $\mathcal I^{GW}$.
\end{theorem}

\subsubsection{Coefficient side - sum rules}
For a probability measure $\mu$ on $\T$ we denote by  $\abold := (\alpha_k)_k$ the sequence of its Verblunsky coefficients.

On the unit circle, the most famous sum rule is the Szeg{\H o} formula:
\begin{equation}
\label{SVsum}
\mathcal K(\lambda_0\, |\ \mu) = -\sum_{k \geq 0} \log(1 - |\alpha_k|^2)
\end{equation}
where, 
 as above $\lambda_0$ is the normalized Lebesgue measure on  $\mathbb T$, whose Verblunsky coefficients are $\alpha_k = 0$ for every $k$.

There are many proofs of (\ref{SVsum}) in \cite{simon05} and a probabilistic proof in \cite{gamboacanonical}.

In the Gross-Witten case, we define, for $|\g| \leq 1$
\begin{align}
\label{defigw}
\mathcal I_\coeff^{\GW} (\al) := 
  \mathcal K(\GW_{\g}\, |\ \lambda_0) - \g \Re \left( \alpha_0- \sum_{k=1}^\infty \alpha_k \bar\alpha_{k-1}\right)+ \sum_{k=0}^\infty - \log (1 - |\alpha_k|^2)\,.
\end{align}
where 
\begin{equation}
\label{defH}
\mathcal K(\GW_{\g}\, |\ \lambda_0) := 
 1 -\sqrt{1-\g^2} +  \log \frac{1+\sqrt{1-\g^2}}{2} 
\,.\end{equation}
The following sum rule was pointed out in \cite{simon05} Theorem 2.8.1 for $\GW_{-1}$  
 and extended in Cor. 5.4 in \cite{GNROPUC}) for $ -1 < \g < 0$, but the proof remains valid for $0 < \g\leq 1$. In \cite{BSZ}, the authors proved the LDP for the coefficient side  when $\g =-1$  by probabilistic arguments, and actually 
this  proof may be extended easily to the case $|\g| < 1$.

\begin{theorem}
\label{coreasy}
Let $\mu$ be a probability measure on $\mathbb{T}$ with Verblunsky coefficients $\abold = (\alpha_k)_{k \geq 0}\in \mathbb D^{\mathbb N}$. Then,
for $0 \leq|\g| < 1$, 
we have
\begin{align}
\label{sumrulegwg}
\mathcal I^{\GW}_\coeff (\abold) = \mathcal I^{\GW}_\meas(\mu)\,. 
\end{align}
\end{theorem}
\begin{remark}
The case $|\g| > 1$ is more complex and involve outliers (see  \cite{GNROPUC} and \cite{Gateway}).
\end{remark}

\section{LDP for perturbations} \label{sec:LDPperturb}
\lb{pert}

We are now in position to state and prove our main results, which are concerned with large deviations of spectral measures of rank-one perturbation of the models introduced in Section \ref{sec:notation}. As advertised during the Introduction, we will always provide two proofs. The first proof, which will be called the {\it direct proof}, uses the fact that the law of the spectral measure of the deformed model is a tilted version of the law of the initial model. The second proof, which will be called the {\it alternative proof}, uses the fact that the Jacobi (resp. Verblunsky) coefficients of the deformed models are simple perturbations of the initial coefficients (in fact, only one parameter is affected).

\subsection{Additive perturbation - Gaussian case}\label{subsec:AddHerm}
For all $n \geq 1$, let us consider
\[W_n = \frac{1}{\sqrt{n}} X_n + A_n,\]
where $X_n$ follows the $\GUE (n)$ distribution and $A_n$ is a rank-one Hermitian deterministic matrix  of size $n \times n$. Since the Gaussian Unitary Ensemble is unitarily invariant, we can assume that  $A_n = \theta uu^*$, where $\theta\in \mathbb R$ and where $u = e_1$ is the first vector of the canonical basis. Let $\mu_\w\sn$ be the spectral measure of the pair $(W_n, u)$. 
It is known (\cite{lee2015edge} Th. 4.6, \cite{noiry2019spectral} Cor. 1) that, as $n \rightarrow \infty$, $\mu_\w\sn$ converges in probability towards the following probability measure:
\begin{align}
\label{defmuwlim}
\mu_{\SC, \theta}(dx) = \frac{\sqrt{(4-x^2)_+}}{2\pi (\theta^2 +1 - \theta x)}\ dx + \left(1 - \theta^{-2} \right)_{+}  \delta_{\theta+ \theta^{-1}}\,.
\end{align}
Our first result establishes a large deviation principle for the sequence of probability measures $(\mu_\w\sn)_{n \geq 1}$.
\begin{theorem}\label{theo:perturbGUE}
The family $(\mu_\w\sn)$ satisfies in $\mathcal M_1^d(\mathbb R)$ (see Cor. \ref{cor42}) the LDP at scale $n$ with good rate function
\begin{align}\label{eq:Rate1}
\mathcal I^W(\mu) = 
\left\{
\begin{array}{lr}
\mathcal K(\SC\, |\ \mu) - \theta m_1(\mu)+ \frac12 \theta^2 + \sum_k \mathcal F_H (E_k^\pm) & \hbox{if} \  \mu \in \Sr_1(-2,2), \\
\infty & \hbox{otherwise}.
\end{array}
\right.
\end{align}
Moreover:
\begin{enumerate}
\item $\mu_{\SC, \theta}$ is the unique minimizer of $\mathcal I^W$, 
\item $\mu_{\SC, \theta}$ is the unique minimizer of $\mathcal I^H_\meas$ under the constraint $m_1(\mu) = \theta$, where we recall that $\mathcal I^H_\meas$ is defined in \eqref{eq:defnImeasGUE}.
\end{enumerate}
\end{theorem}

\begin{proof}

We first prove \eqref{eq:Rate1} using two different arguments.
\medskip

A) Direct proof. If  $X_n$ has the  $\GUE(n)$ distribution (see (\ref{defGUE})), the distribution of $W = n^{-1/2}X_n+ \theta uu^*$ is 
\[\mathbb P\sn_\theta (dW) = \frac{1}{\mathcal Z_n} \exp \left( -\frac{n}{2} \tr ((W-\theta uu^*)(W^\star - \theta uu^*)) \right)   dW. \]
But
\begin{align*}\nonumber
\tr ((W-\theta uu^*)( & W^\star  - \theta uu^*))\\ 
&= \tr (WW^\star) - \theta \left(\tr(Wuu^*) + \tr (uu^* W^\star)\right) + \theta^2 \tr (uu^*) \\
&= \tr (WW^\star) - 2 \theta u^* W u + \theta^2,
\end{align*}
which allows us to rewrite
\[\mathbb P_\theta\sn (dW) =
\exp \left( -\frac{n}{2} (\theta^2 - 2 \theta u^* W u) \right) \ \mathbb P_0\sn (dW)\,. \]
Since $u=e_1$, one has $u^* W u = W_{11} = m_1 (\mu_\w\sn)$, which yields
\begin{align}
\notag
\mathbb P_\theta\sn( \mu_\w\sn \in d\mu) &= 
\frac{\exp n \Psi(\mu)}{\mathbb P_0\sn\left(\exp n \Psi(\mu_\w\sn)\right)}
\mathbb P_0\sn( \mu_\w\sn \in d\mu),
\end{align}
where
\[
\Psi(\mu) = \theta m_1 (\mu) 
\]
and
\begin{align}
\label{gamma}
\mathbb P_0\sn\left(\exp n \Psi(\mu_\w\sn)\right) = \exp \frac{n\theta^2}{2}\,.
\end{align}
It remains to  apply Varadhan's lemma (see \cite{ellis1985entropy} Th II.7.2,  \cite{demboz98} Th 4.3.1 or \cite{deuschel2001large} Exercise 2.1.24):
\begin{itemize}
\item $\Psi$ is continuous with respect to $\mathcal M_1^d$
\item  the uniform exponential integrability condition is satisfied since
(\ref{gamma}) implies
\[
 \frac{1}{n} \log \mathbb P_0\sn[ \exp \gamma n m_1 (\mu_\w\sn)]
 =   \frac{\gamma^2 }{2}\,.\]
\end{itemize}
The rate function is then $\mathcal I^H_\meas - \Psi + \frac{\theta^2}{2}$ since
 (\ref{gamma}) provides the infimum term.
\medskip

B) Alternative proof. Fix $n\geq 1$. A consequence of the tridiagonal representation of the $\GUE(n)$ of Dumitriu and Edelman (\cite{agz} Sec. 4.5) is that $\mu_\w\sn$ is the spectral measure of the pair $(J_n,e_1)$, where $J_n$ is the following random Jacobi matrix:
\[  J_n \sim 
\left(
\begin{array}{ccccc}
\mathcal{N}( 0, \frac{1}{n}) + \theta       & \frac{1}{\sqrt{n}} \chi_{2(n-1)} &                  &               &                 \\
\frac{1}{\sqrt{n}} \chi_{2(n-1)} & \mathcal{N}( 0, \frac{1}{n})        & \frac{1}{\sqrt{n}} \chi_{2(n-2)}   &               &                 \\
               & \frac{1}{\sqrt{n}} \chi_{2(n-2)}   & \mathcal{N}( 0, \frac{1}{n})          & \ddots        &                 \\
               &                &  \ddots          & \ddots        & \frac{1}{\sqrt{n}} \chi_{2}      \\
               &                &                  & \frac{1}{\sqrt{n}} \chi_{2}      & \mathcal{N}( 0, \frac{1}{n})
\end{array}
\right). \] 
Here, the matrix $J_n$ is symmetric and up to this symmetry, its coefficients are independent. Note that this corresponds to the usual tridiagonalisation of the $\GUE(n)$ except for the addition of the parameter $\theta$ to the $(1,1)$ coefficient. 

Fix $\mu$ a measure on $\mathbb R$ with Jacobi parameters $\a = (a_n)_{n\geq 1}$ and $\mathfrak{b} =(b_n)_{n \geq 1}$. Then, using a projective method and  the independence of the coefficients of $J_n$, as in \cite{gamboa2011large}  we see that  $(\mu_\w\sn)$ satisfies the LDP in $\mathcal M_1^d$ equipped with $\mathcal T_m$, with 
the rate function 
 given by
\begin{align}\label{forinf}
\mathcal I^W_{\coeff}(\a , \mathfrak{b} )  &= \frac{1}{2}(b_1-\theta)^2 +\frac{1}{2}\sum_{k \geq 2} b_k^2 + + \sum_{k\geq 1}  G(a_k^2)\\
\lb{forsup}
&=  \frac12 \theta^2 - b_1\theta + \mathcal I^H_{\coeff}( \a , \mathfrak{b})\,.
\end{align}
But, by Theorem \ref{sumruleg},
\[\mathcal I^H_{\coeff}(\a, \mathfrak{b} ) =\mathcal I^ H_\meas(\mu) = \mathcal K(\SC\, |\ \mu) + \sum_k \mathcal F (E_k^\pm)\,.\] 
Besides,  $b_1 = m_1 (\mu)$, so that the random measure $\mu_\w\sn$ satisfies the LDP at scale $n$ with rate function
\begin{align}
\label{I=IH}\mathcal I^W(\mu) =  \mathcal I_{\coeff}^W(\a, \mathfrak{b}) = \mathcal K(\SC\, |\ \mu) - \theta m_1 (\mu)
+ \frac12 \theta^2 + \sum_k \mathcal F (E_k^\pm)\,.\end{align}

\bigskip

We now turn to the proofs of (1) and (2).

\noindent (1) The infimum can be looked from the coefficient side, i.e. from (\ref{I=IH}) and (\ref{forinf}), and is given by:
\[\jac(\operatorname{argmin} \mathcal I^W)= \begin{pmatrix}
\theta, & 0,& 0,& \cdots\\
1, & 1,& 1, &\cdots
\end{pmatrix}. \]
Using section \ref{freeM}, we deduce that $\jac(\operatorname{argmin} \mathcal I^W) = \jac(\tilde\mu_{-\theta,0}) $, namely $\operatorname{argmin} \mathcal I^W = \mu_{\SC, \theta}$.

\smallskip

\noindent (2) If $m_1(\mu) = \theta$, we deduce from (\ref{forsup}) and the sum rules that
\[\mathcal I^H_{\meas}(\mu) = \mathcal I^W_{\meas} (\mu) + \frac{\theta^2}{2}\geq \frac{\theta^2}{2}\,,\]
with equality if and only if $\mu= \mu_{\SC, \theta}$.

\end{proof}

\begin{remarks}
\begin{enumerate}
\item
Let us first observe that we can check by hands that the minimum of $\mathcal I^W$ is zero. Indeed, we have
\begin{align*}
\mathcal F_H(\theta + \theta^{-1}) = \frac{\theta^2 - \theta^{-2}}{2} - 2 \log \theta
\end{align*}
and besides  $\mathcal F_H$, which is the rate function of the top eigenvalue,  is also given by (\cite{agz} Th. 2.6.6
\footnote{There is a mistake in \cite{agz} p.81 see http://www.wisdom.weizmann.ac.il/~zeitouni/cormat.pdf}) 
\[ \mathcal F_H (\theta + \theta^{-1}) = \frac{(\theta + \theta^{-1})^2}{2} - 2 \int \log (\theta + \theta^{-1} -x) \SC(dx) -
1, \]
Since
\begin{align*} 
\mathcal K(\SC\, |\ \mu_{\SC, \theta}) 
&= \int \log (1 + \theta^2  - \theta x) \SC(dx)\\  
&= \log \theta + \int \log (\theta + \theta^{-1} -x) \SC(dx)\,,
\end{align*}
we deduce that 
$\mathcal K(\SC\, |\ \mu_{\SC, \theta}) + \mathcal{F} ( \theta + \theta^{-1}) = \frac{\theta^2}{2}$ and :
\begin{align*}
\mathcal I_{\meas}^W ( \mu_{\SC, \theta}) 
&= \mathcal K(\SC\, |\ \mu_{\SC, \theta}) + \mathcal{F} ( \theta + \theta^{-1})  - \theta m_1(\mu_{\SC, \theta}) + \frac{\theta^2}{2}  \\
&= \frac{\theta^2}{2}  - \theta^2 + \frac{\theta^2}{2} = 0.
\end{align*}

\item The fact that $\mu_{SC,\theta}$ is the only minimizer of $\mathcal I^W$ allows to retrieve the convergence of $\mu_\w\sn$ towards $\mu_{\theta, \SC}$, and actually to strengthen the convergence in probability into an almost sure convergence.
\end{enumerate}
\end{remarks}

\subsection{Multiplicative perturbation}\label{subsec:MultPert}

For all $n \geq 1$, let us consider
\[S_n = \frac{1}{n} \Sigma_n^{1/2}L_n\Sigma_n^{1/2}\]
where $\Sigma_n = \diag (\theta, 1, \dots)$ with $\theta >0$. It is known (as a consequence of the anisotropic local laws derived in \cite{knowles-yin}) that the sequence of measure $(\mu_\w\sn)_{n \geq 1}$ has a limit. An explicit computation of the limiting measure $\mu^{L,\theta}$ can be performed as in \cite{noiry2019spectral}, and we get\footnote{This is the same measure as
$\mu_{\MP, \tau^{-1},\theta}$ in \cite{noiry2019spectral} up to a little change, due to the convention on the definition of sample covariance matrix.} 
\begin{align}
\label{muL}
\mu^L(dx) = \frac{\sqrt{4\tau- (x- (1+\tau))^2}}{2\pi x\left((\theta+\tau -1) + x (\theta^{-1}-1)\right)}dx + u\delta_0  + v\delta_w,
\end{align}
with 
\[u = \frac{(\tau-1)_+}{\theta+\tau-1} , \,  v = \frac{\tau\left((\theta-1)^2 -\tau\right)_+}{(\theta -1)(\theta+\tau-1)}, \,  w=- \frac{\theta+\tau-1}{\theta^{-1} -1}\,.\]
Here, we will restrict the setting to the case where $n/N \to \tau <1$, that is to the case where $\mu^L$ does not have a mass at zero. In this context, we obtain a large deviation principle for the family of spectral measures $(\mu_\w\sn)_{n\geq 1}$ associated to the pairs $(S_n,e_1)$.
\begin{theorem}
\lb{corLUE}
The family $(\mu_\w\sn)$ satisfies the LDP at scale $n$ in $\mathcal M_1^d$ with good rate function
\begin{equation}
\label{IWL}
\mathcal I^S(\mu) =
\left\{
\begin{array}{lr}
\mathcal K(\MP_\tau\, |\ \mu) + \frac{\theta^{-1} -1}{\tau} m_1(\mu) + \frac{1}{\tau}
\log \theta
  + \sum\limits_{k} \mathcal F^\pm_L (E_k^\pm) & \hbox{if} \  \mu \in \Sr_1(\tau^-, \tau^+)\\
\infty & \hbox{otherwise}.
\end{array}
\right.
\end{equation} 
Moreover,
\begin{enumerate}
\item $\mu^{L,\theta}$ is the unique minimizer of $\mathcal I^S$,

\item $\mu^{L,\theta}$ is the unique minimizer of $\mathcal I^L$ under the constraint 
$m_1(\mu) = \theta$.
\end{enumerate}
\end{theorem}

\begin{proof}

We first provide two proofs of \eqref{IWL}.
\medskip

A) Direct proof. Let  $L_n$ be a random $n\times n$ matrix following the $\LUE_N(n)$ distribution (see \eqref{defLUE}), and $\Sigma_n$ a Hermitian positive $n\times n$ matrix. Then, the distribution of $S_n = \Sigma_n^{1/2}L_n\Sigma_n^{1/2}$ is 
\begin{equation*}
\mathbb Q_\theta\sn  (dS) = \frac{1}{\mathcal Z_n}
(\det S)^{N-n}(\det \Sigma)^{-N} \exp \left(- N\tr\!\ \Sigma^{-1}S\right) \ dS. 
\end{equation*}
In our case, $N= \tau^{-1}n$ and $\Sigma_n =\diag(\theta, 1, \dots, 1)$, so that we have $\det \Sigma_n = \theta$ and:
\[
\tr (S_n \Sigma_n^{-1}) = \theta^{-1} (S_n)_{11} + \sum_{k\geq 2} (S_n)_{kk} = (\theta^{-1}-1) (S_n)_{11} + \tr\!\ S_n\,,
\]
which allows us to rewrite,
\begin{align*}
\mathbb Q_\theta\sn  (dS) 
=  
\theta^{-n \tau^{-1}} 
 \exp\left( n \tau^{-1}
(1- \theta^{-1}) S_{11}
 \right)  \  
 dS\,.
\end{align*}
Moreover, since $\mu_\w\sn$ is the spectral measure associated to the pair $(S_n,e_1)$, we have $S_{11} = m_1(\mu_\w\sn)$, which implies that
\begin{align*}
\mathbb Q_\theta\sn( \mu_\w\sn \in d\mu) = \frac{\exp n \Phi(\mu)}{\mathbb Q_1\sn\left(\exp n \Phi(\mu_\w\sn)\right)}\mathbb Q_1\sn( \mu_\w\sn \in d\mu),
\end{align*}
where
\[ \Phi(\mu) = \tau^{-1} (1 - \theta^{-1})m_1(\mu) \]
and
\begin{align}
\label{gammaL}
\mathbb Q_1\sn\left(\exp n \Phi(\mu_\w\sn)\right) = \theta^{n\tau^{-1}}\,.
\end{align} 
In order to apply Varadhan's Lemma, let us check
 the uniform exponential integrability condition. 
From (\ref{gammaL}) we have
\begin{align}
\label{cstwish}
\mathbb Q_1\sn (\exp \varphi n\tau^{-1}m_1(\mu_\w\sn)) = \left(1 - \varphi\right)^{-n\tau^{-1}}\,,
\end{align}
 for all $\varphi <1$. 
Therefore 
\[\frac{1}{n} \log \mathbb Q_1\sn \exp \gamma n \Phi(\mu_\w\sn) = -\tau^{-1} \log(1-\gamma(1-\theta                                             ^{-1}))
\,,\]
for any  $\gamma\in (1, (1-\theta^{-1})^{-1}_+)$,
 the uniform integrability condition is satisfied. The rate function is $\mathcal I^L_\meas -\Phi + \tau^{-1} \log \theta$  
, since (\ref{cstwish}) provides the constant term.
\medskip

B) Alternative proof. 

Fix $n\geq 1$. A consequence of the tridiagonal representation of the $\LUE(n)$ of Dumitriu and Edelman is that $\mu_\w\sn$ is the spectral measure of the pair $(J_n,e_1)$, where $J_n = B_n B_n^\star$ with:
\[  B_n \sim  \frac{1}{\sqrt{2N}}
\left(
\begin{array}{ccccc}
\chi_{2N}      &  &                  &               &                 \\
\sqrt{\theta} \cdot \chi_{2(n-1)} & \chi_{2(N-1)}    &   &               &                 \\
               & \chi_{2(n-2)}  & \chi_{2(N-2)}        &       &                 \\
               &                &  \ddots          & \ddots        &      \\
               &                &                  &  \chi_{2}      & \chi_{2(N-n+1)}
\end{array}
\right). \] 
Here, the matrix $B_n$ is bidiagonal and its coefficients are independent. Note that this corresponds to the usual bidiagonal matrix of the $\LUE(n)$ except for the addition of the multiplicative factor $\sqrt{\theta}$ to the $(1,2)$ coefficient. Using the parameters system (\ref{zerl1}), we deduce that the transformation $L_n \mapsto S_n$ changes the first coefficient $z_1$ into $z_1' =\theta z_1$ and does not change the other parameters. Since the rate function for $z_1$ is $\tau^{-1}G(z)$ with
$G(z) = z-1 - \log z$, the rate function for $\theta z_1$ is $\tau^{-1}G(z/\theta)$. 
Let $\mu$ be a positive measure on $[0,\infty)$ with $\mathfrak{z}$-parameters $(z_i)_{i \geq 0}$. Then, using a projective method and  the independence of the coefficients of $J_n$ as in \cite{gamboa2011large}, 
we see that 
the LDP on the coefficient side is given by:
\begin{align*}
\mathcal I^S_{\coeff}(\z) &=  \tau^{-1}\left[G(z_1/\theta) - G(z_1)\right] +\mathcal I^L_{\coeff}(\z) \\
\lb{forsupL}
&= \tau^{-1}(\theta^{-1} - 1) z_1 + \tau^{-1} \log \theta +\mathcal I^L_{\coeff}(\z)\,.
\end{align*}
But by the sum rule \eqref{srL},
\begin{align}\mathcal I^L_{\coeff}(\z)  =  \mathcal K(\MP\, |\ \mu) + \sum_k \mathcal F_L (E_k^\pm).
\end{align} 
Moreover, $z_1 = b_1 = m_1(\mu)$, so that our random measure satisfies the LDP with rate function
\begin{align}
\mathcal I^W_{\meas}(\mu) =  \mathcal K(\MP\, |\ \mu) + \tau^{-1}(\theta^{-1} - 1) m_1(\mu)+  \tau^{-1} \log \theta + \sum_k \mathcal F_L (E_k^\pm).
\end{align}
\medskip

We now turn to the proof of (1) and (2).

(1) The minimizer of $\mathcal{I}^S$ can be looked from the coefficient side and is given by the following $\mathfrak{z}$-parameters:
\begin{align}
z_1 = \theta, z_{2k-1} =1 \ , \ k \geq 2\ , \
z_{2k} = \tau \ , \ k \ \geq 1\,.
\end{align}
Owing to (\ref{zerl1}), it corresponds to the following Jacobi coefficients:
\begin{align}
\jac(\operatorname{argmin} I^W)= \begin{pmatrix}
\theta, & 1 + \tau,& 1+\tau,& \cdots\\
\sqrt{\theta\tau}, & \sqrt \tau,& \sqrt \tau, &\cdots
\end{pmatrix}.
\end{align}
By Lemma \ref{affine}, we deduce that
\begin{align}
\jac\left(T_{\sqrt{\theta\tau}, \theta}(\operatorname{argmin} I^W)\right)=
\jac(\mu_{b,c})
\end{align}
where
\begin{align}
b = \frac{1+\tau-\theta}{\sqrt{\theta\tau}}\ , \ c= \frac{1-\theta}{\theta}\,.
\end{align}
Coming back to our distribution, we find  the expression given in (\ref{muL}) for $\mu^L$. For $\theta >1$ (resp. $\theta< 1$), it is the free binomial (resp. free Pascal) distribution (see Section \ref{freeM}).

(2) The condition $m_1(\mu) = \theta$ rewrites $\z_1 = \theta$. Combining (\ref{forsupL}) and the sum rule \eqref{srL}, we deduce that
\[\mathcal I^L (\mu) = \mathcal I^S(\mu) + \tau^{-1}(\theta-1- \log \theta)\geq \tau^{-1}(\theta-1- \log \theta)\,,\]
with equality if and only if $\mu = \mu^L$.
\end{proof}

\subsection{Perturbations of Unitary Matrices}\label{subsec:GrossPert}
To the best of our knowledge, there is only one  type of perturbation of unitary matrices which was studied in relation with Verblunsky (for short ``V") coefficients. 
If $U_n \in \mathbb U(n)$ and $e= e_1$ is cyclical, let as usual $(\alpha_k)_{k \geq 0}$ be the V-coefficients of the pair
$(U_n, e)$. 
Now, for any fixed element $e^{\ii \varphi} \in \mathbb T$, we define
\begin{align}
W_n = U_n Q_n \ , \ \text{with} \ Q_n = I_n + (e^{\ii \varphi}- 1) e\langle e, \cdot \rangle \ .
\end{align}
Such a rank-one perturbation has been considered in Sections 1.3.9, 1.4.16, 3.2, and 4.5  of \cite{simon05}, 10.1, A.1.D and A.2.D of  \cite{Simon2}, see also \cite{simon-rk1}.

If $\mu$ is the spectral measure of the pair $(U_n, e)$ let us denote by $\tau_{e^{\ii \varphi}}\mu$ the spectral measure of the pair $(W_n, e)$. A usual tool for the study of a measure $\mu$ on $\mathbb T$ is its Caratheodory transform, which  is the analog of the Stieltjes transform, defined by
\begin{align}
F_\mu(z) = \int_0^{2\pi} \frac{e^{\ii \theta} +z}{e^{\ii \theta} -z} d\mu(\theta)\,.
\end{align}
Conversely, if $d\mu = w(\theta) d\lambda_0(\theta) + d\mu_s$, 
 then
\begin{align}
\label{recover}w(\theta) = \lim_{r\uparrow 1} \Re F_\mu (re^{\ii \theta})\end{align} 
and $\mu_s$ is supported by  $\{\theta : \lim_{r\uparrow 1} \Re F_\mu (re^{\ii \theta}) = \infty\}$
(see \cite{simon05} (1.3.31)). 

The mapping 
 ($\mu \mapsto \tau_{e^{\ii \varphi}}\mu$) 
 gives at the level of Caratheodory transform  :
\begin{align}
\label{alex0}
F_{\tau_{e^{\ii \varphi}}\mu}=  \frac{(1 - e^{-\ii \varphi}) + (1 + e^{-\ii \varphi}) F}{(1 + e^{-\ii \varphi}) + (1 -e^{-\ii \varphi}) F}\,,
\end{align}
 (see \cite{simon05}(1.3.90)), which implies, by the Schur recursion, the remarkable relation:
\begin{align}
\label{changealpha}
\alpha_k (\tau_{e^{\ii \varphi}}\mu)= e^{-\ii \varphi}\alpha_k(\mu) \ , \ (k \geq 0)\,.
\end{align}
When $\varphi$ is varying, it generates the so-called Aleksandrov family of measures. 

\noindent In particular, if $\mathbb R_\varphi\sn$ (resp. $\mathbb R_0\sn$) denotes the distribution of $W_n$ (resp. $U_n$), we have 
\begin{align}
\label{5.23}
\mathbb R_\varphi\sn(\mu_\w\sn \in d\mu)= \mathbb R_0\sn \left(\mu_\w\sn \in d(\tau_{e^{-\ii \varphi}}\mu)\right)\,.
\end{align}
Here is  our theorem which establishes a large deviation principle for the sequence of spectral measures associated to the pairs $(W_n,e)$.

\begin{theorem}
\lb{corGW}
Assume $|\g|\ \leq 1$.
\begin{enumerate}
\item
The family of distribution of random measures $(\mu_\w\sn)$ under $\mathbb G\mathbb W_\g(n)$ satisfies the LDP on $\mathcal M_1(\mathbb T)$, at scale $n$  with good rate function
\begin{align}
\label{rfGW}
\mathcal I^W (\mu) =  \mathcal I^{GW}(\mu)   -\g \Re \left( ( e^{-\ii \varphi}-1)m_1(\mu)
\right)\,,
\end{align}
where $\mathcal I^{GW}$ has been defined in Theorem \ref{appliGW}.
\item  
The unique minimizer of $\mathcal I^W$ is 
$\mu^\varphi = \tau_{e^{-\ii \varphi}}(\GW_{\g})$ and 
\begin{align}
\label{equi34}
d\mu^\varphi(\theta) 
=  \frac{1}{2\pi} \frac{1  +\g \cos \theta}{1  - 2\g \sin \frac{\varphi}{2}\sin \left( \theta - \frac{\varphi}{2}\right) +\g^2 \sin^2 \frac{\varphi}{2}}\ d\theta\,.
\end{align}
\item $\mu^\varphi$ is the unique minimizer of $\mathcal I^{\GW}_\meas$ under the constraint $m_1(\mu) = \frac{\g}{2} e^{\ii \varphi}$.
\end{enumerate}
\end{theorem}

\proof

\noindent (1) A) Direct  proof.  

From (\ref{GW0}) we deduce
\begin{align}
\mathbb R_\varphi\sn(dW) =  \frac{1}{\mathcal Z_n} \exp  \frac{n\g}{2} \tr  ( WQ_n^{-1} + (WQ_n^{-1})^\star)\ dW\,.
\end{align}
But
\begin{align}\tr\!\ (WQ_n^{-1}) = 
 \tr\!\ W + (e^{-\ii \varphi} -1) W_{11}\ , \ 
 \tr\!\ (WQ_n^{-1})^\star = \tr\!\ W^\star + (e^{\ii \varphi} -1) \bar W_{11}
\end{align}
 so that
\begin{align}
\mathbb R_\varphi\sn(dW) = \exp  n\g\left( \Re\!\ (e^{-\ii \varphi}-1) W_{11})\right)\mathbb R_\varphi\sn(dW)
\end{align}
and since $W_{11} = m_1(\mu_\w\sn)$ we get 
\begin{align}
\label{noconstant}
\mathbb R_\varphi\sn(\mu_\w\sn \in d\mu)) = \exp n\g \Re \left(e^{-\ii \varphi} -1)m_1(\mu)\right)\mathbb R_0\sn(\mu_\w\sn \in d\mu))\,.
\end{align}
This yields (1) by application of Varadhan's lemma without integrability  condition
 since $m_1(\mu_\w\sn)  \in \mathbb D$. Notice that due to the form of  (\ref{noconstant}), there is no constant term in the rate function.
\medskip

(1\noindent ) B) An alternative  proof 

Under $\mathbb G\mathbb W_{\g}(n)$, the rate function for the LDP of the V-coefficients is $\mathcal I_{\coeff}^{\GW}$ given by (\ref{defigw}). 
After a pushing forward by (\ref{5.23}) 
  the new rate function on the coefficient side becomes  
\begin{align}
\lb{537}\mathcal I_{\coeff}^W(\abold) = 
\mathcal I_{\coeff}^{\GW}(e^{\ii \varphi} \abold) = \mathcal I_{\coeff}^{\GW}(\abold)  
- \g \Re \left(\alpha_0 (e^{\ii \varphi}-1)\right)\,.\end{align}
Coming back to the sum rule and using
 $\alpha_0(\mu) = \bar m_1 (\mu)$ we get (\ref{rfGW}).

\noindent(2) From (\ref{5.23}), we have
\begin{align}
\mathcal I^W(\mu) = \mathcal I^{\GW} (\tau_{e^{-\ii \varphi}}(\mu))\,.
\end{align} 
Therefore, the rate function $\mathcal I^W$ has a unique minimum at
\begin{align}
\label{defmuphi}\mu^\varphi := \tau_{e^{\ii \varphi}}(\GW_{\g})\,.\end{align}

The Caratheodory transform of the  equilibrium measure is (\cite{simon05} p.86)
\[F(z) = 1  +\g z\,,\]
so that, using (\ref{defmuphi}), (\ref{alex0}) and (\ref{recover}), 
we find the density 
(\ref{equi34}). Moreover there is no extra mass since $F^\varphi$ has no pole on $\mathbb T$.

We could also have applied formula (3.2.96) in \cite{simon05}, which states that if $\mu= w(\theta)d\lambda_0(\theta) + d\mu_s (\theta)$, then the density $\tilde w$ 
 of $\mu^\varphi$ 
 is given by
\begin{align}
\label{5.29}
\tilde w(\theta) = \frac{w(\theta)}{|\cos\frac{\varphi}{2} + \ii \sin\frac{\varphi}{2} F(e^{\ii \theta})|^2}.
\end{align}
\smallskip

\smallskip

\noindent (3) If $m_1(\mu) =  e^{\ii \varphi}\g/2$, we deduce from (\ref{rfGW}) 
 that
\[\mathcal I^{GW}(\mu) = \mathcal K(\GW_{\g}\, |\ \mu) = \mathcal I^W (\mu) 
+\frac{\g^2}{2}(1 - \cos\varphi) \geq \frac{\g^2}{2}(1-\cos\phi)\,,\]
with equality if and only if $\mu = \mu^\varphi$.

\section{Generalizations}\label{sec:generalization}
In this section, we discuss two possible generalizations of our considerations. The first one concerns the rank-one perturbations of invariant models with general potentials and the second one deals with a matricial version of our results. In each case, for the sake of clarity and to avoid numerous repetitions, we will only treat in details the Hermitian setting.
\subsection{General potential}
\paragraph*{{\bf Additive perturbation.}}
Let $V$ be a convex polynomial potential of even degree $2d$ with positive leading coefficient:

\begin{align}
\label{defV}  V(x) = a_{2d} x^{2d} + \cdots, \quad a_{2d} > 0\,.  
\end{align}
 Let $\mathbb P_0\sn$ be the invariant measure on the set of $n \times n$ Hermitian matrices given by:
\begin{align}
\label{P0V}
\mathbb P_0\sn (dH) = \frac{1}{\mathcal Z_n} \exp \big( -n \tr\!\ V(H)\big)  dH.
\end{align}
Under our assumptions on $V$, this model has a unique equilibrium measure $\mu_V$, which is the almost-sure limit of the empirical spectral measures. Moreover, $\mu_V$ is supported by a single interval $[a_V,b_V]$ and has a density of the form:
\[  \mu_V(d x) = \frac{1}{\pi} r(x) \sqrt{(b_V-x)(x-a_V)} \mathbf{1}_{[a_V,b_V]}(x) dx,  \]
where $r$ is a polynomial of degree $2d-2$ with nonreal zeros (see for example Proposition 3.1 and Equation (2.8) of \cite{johansson1998fluctuations}).

As in Section \ref{subsec:AddHerm}, we are interested in the following additive rank-one perturbation of the model:
\[  W_n := H_n + \theta e_1 e_1^T, \quad H_n \sim \mathbb P_0\sn.  \]
Denoting $\pi = e_1 e_1^T$, we see from (\ref{P0V}) that the distribution of random matrix $W_n$ is :
\begin{align}
\label{PV}
\mathbb P_\theta\sn (dW) := \exp \big( -n  \left[ \tr\!\ V(W- \theta \pi) - \tr\!\ V(W)\right] \big)\  \mathbb P_0\sn (dW) .
\end{align}
Let $\mu_\w\sn$ be the spectral measure associated to the pair $(W_n,e_1)$. In order to compute the distribution of $\mu_\w\sn$, we need the following lemma, whose proof is postponed to the end of this section.
\begin{lemma}
\label{mono}
There exists a polynomial $Q_V$ in $2d$ variables such that, for all Hermitian matrix $M$, 
\[ \tr\!\ V(M - \theta \pi) - \tr\!\ V(M) = Q_V(\theta, M_{11}, (M^2)_{11}, \ldots, (M^{2d-1})_{11})\,.\]
\end{lemma}
\begin{remark}
Although a concise formula for $Q_{V}$ in function of $V$ seems out of reach, let us give two simple examples:
\begin{itemize}
\item when $V(x)=x^2$, $Q_V = \theta^2 - 2 \theta M_{11}$,
\item when $V(x) = x^4$, $Q_V = \theta^4 - 4 \theta^3 (M^3)_{11} + 4\theta^2 (M^2)_{11} + 2 \theta^2 M_{11}^2 - 4 \theta (M^3)_{11}$.
\end{itemize}
\end{remark}
\noindent With the notation of Lemma \ref{mono}, we have that
\begin{align*}
\mathbb P_\theta\sn (\mu_\w\sn \in d\mu) = \exp \big( -n Q_{V}(\theta, m_1(\mu_\w\sn), \dots, m_{2d-1}(\mu_\w\sn)) \big) \mathbb P_0\sn (\mu_\w\sn \in d\mu),
\end{align*}
where we recall that $m_i(\mu_\w\sn)$ stands for the $i$-th moment of $\mu_\w\sn$. We also need the following observation, whose proof is postponed to the end of this section.
\begin{lemma}
\label{Vara2}
If $V$ is a convex polynomial of even degree,
\begin{align}\sup_n \frac{1}{n} \log \mathbb E \exp n \gamma\left(\tr\!\ V(M- \theta \pi) - \tr\!\ V(M)\right) < \infty
\end{align}
\end{lemma}
This exponential integrability allows an application of  
 Varadhan's Lemma, which gives the following result.
\begin{theorem}
\label{6.3}
The sequence of probability measures $(\mu_\w\sn)_{n \geq 1}$ satisfies a large deviations principle at scale $n$ with good rate function:
\begin{align}
\mathcal I^W (\mu) =\mathcal J(\mu)  - \inf_\nu \mathcal J(\nu)\,,
\end{align}
where
\begin{align}
\mathcal J(\mu) := \mathcal K(\mu_V\, |\ \mu) -Q_{2d}(\theta, m_1(\mu_\w\sn), \dots, m_{2d-1}(\mu_\w\sn))+ \sum_k \mathcal F_H(E_k^\pm)\,.
\end{align}
\end{theorem}
We now turn to the proofs of Lemmas \ref{mono} and \ref{Vara2}.

\begin{proof}[Proof of Lemma \ref{mono}]

It is enough to check the assertion when $V$  a monomial $V(x) = c x^r$.
The matrix $(M- \theta \pi)^r$ is the sum of $2^r$ products of elements which are $M$ or $\theta\pi$. Since $\pi$ is a projection, $\pi^k = \pi$ for every $k \geq 1$, hence the products involved in $(M+ \theta \pi)^r - M^r$ are of the form 
\begin{enumerate}
\item
$ \theta^j \pi M^{a_1}\pi \cdots \pi M^{a_i}\pi$
\item
$ \theta^j \pi M^{a_1}\pi \cdots\pi M^{a_i}$
\item
$ \theta^j M^{a_1}\pi \cdots\pi M^{a_i}$
\item
$ \theta^j M^{a_1}\pi \cdots\pi M^{a_i}\pi$\,.
\end{enumerate}
It is clear that the first expression is exactly $\theta^j (M^{a_1})_{11}\cdots (M^{a_i})_{11}$. The three other ones can be reduced to the first type:  
since $\tr AB = \tr BA$, we can write
\begin{itemize}
\item
$ \tr (\pi M^{a_1}\pi ...\pi M^{a_i}) = \tr (\pi^2M^{a_1}\pi \cdots \pi M^{a_i})=\tr (\pi M^{a_1}\pi \cdots\pi M^{a_i}\pi)$
\item
$\tr (M^{a_1}\pi ...\pi M^{a_i})= \tr  (M^{a_1+ a_i}\pi \cdots\pi)= \tr  (M^{a_1+ a_i}\pi \cdots\pi^2)= \tr (\pi  (M^{a_1+ a_i}\pi ...\pi)$
\item
$\tr (M^{a_1}\pi \cdots\pi M^{a_i}\pi) = \tr (M^{a_1}\pi \cdots\pi M^{a_i}\pi^2) = \tr (\pi M^{a_1}\pi \cdots\pi M^{a_i}\pi)$
\end{itemize}
and since $\pi M^k\pi = (M^k)_{11} \pi$, we get the result.
\end{proof}

\begin{proof}[Proof of Lemma \ref{Vara2}]

Let us denote $\ell := \max \{|\lambda_{\max}|, |\lambda_{\min}|\}$. Combining Lemma \ref{mono} and the fact that for all $k \geq 0$, $(M^k)_{11} \leq \ell^k$, we deduce that $\tr\!\ V(M+ \theta \pi) - \tr\!\ V(M)$ is bounded by $C \ell^{2d-1}$, for some constant $C$ only depending on $V$ and $\theta$.

Therefore, it is enough to check that $\sup_n n^{-1} \log \mathbb E \exp C n \ell^{2d-1} <  \infty$. This fact is a direct consequence of the following rough large deviations estimate :
there exists $C' > 0$ such that for every $x> 0$ large enough $\mathbb P(\ell > x) \leq e^{-nC' V(x)}$
(see \cite{pastur2011eigenvalue} Theorem 11.1.2, a precise rate function is given in \cite{borotgui2013} Prop. 2.1).
The proof is ended recalling that $V$ is given by (\ref{defV}).
\end{proof}

\subsection{Matricial spectral measures}
Let $E$ be $\T$ or $\R$ and $r$ a positive integer.. 
A matrix measure $\Sigma=(\Sigma_{i,j})$ of size $r\times r$   on $E$
 is a matrix of signed complex measures, such that for any Borel set 
$A\subset E$, $\Sigma(A)= (\Sigma_{i,j}(A))\in \mathcal H_p$ is (Hermitian and) non-negative definite. A matrix measure on $E$ 
 is  a probability matrix measure normalized, if $\Sigma(E)=\mathbf 1$. We denote by $\mathcal{M}_{r,1}(E)$ the set of  $r\times r$ probability matrix measures with support in $T \subset E$.
 
Given a matrix $M$ and a $r$-tuple $(u_1,\ldots,u_r)$ of unit vectors that are orthogonal, 
 we define the matricial spectral measure $\nu^M := (\nu_{ij}^M)_{1 \leq i,j \leq r}$ as the only element of $\mathcal M_{r,1}$ such that, for all $i,j \in \{1, \ldots, r\}$ and all
 $ k \geq 0$ 
 \[ \langle u_i, M^k u_j \rangle = \int_E x^k \mathrm{d} \nu_{ij}^M(x)\,.   \]
In other words
\[  \left((M^k)_{ij}\right)_{1 \leq i,j \leq r} = \left(  \int x^k \mathrm{d}\nu_{ij}^M(x) \right)_{1 \leq i,j \leq r}.  \]
We will denote by $\mathbf{m}_k = \mathbf{m}_k ( (\nu_{ij}^M))$ the right-hand side of the above equality.
Note that when $r=1$, we retrieve the previously considered spectral measure associated to the pair $(M,u_1)$. Interestingly, our method also applies to the study of matricial spectral measures of perturbations of the invariant models described in Section \ref{sec:notation}. In the following, we will always assume that $u_1, \ldots, u_r$  are the first $r$ vectors $e_1, \ldots, e_r$ of the canonical basis. Analogously to Sections \ref{subsec:AddHerm}, \ref{subsec:MultPert} and \ref{subsec:GrossPert}, our results rely on former large deviations principles obtained for the unperturbed models.

In order to state them, we first need to introduce 
some notations. Let $\Sigma \in \mathcal{M}_{r,1}(E)$ be a \emph{quasi-scalar} measure, which means that   
$\Sigma = \sigma \cdot \bdone $ 
 where $\sigma \in \mathcal{M}_{1}(E)$ is a scalar probability measure
 and $\bdone$ is the $r\times r$ identity matrix. 
 Let $\mubold \in \mathcal{M}_{r,1}(E)$. 
 We say that 
 $\mubold$ 
  is absolutely continuous (a.c. for short) with respect to $\sigma$ $(\mubold\ll \sigma)$ if each entry of $\mubold$ is a.c. with respect to $\sigma$. In this case there is a 
 Lebesgue decomposition
\[\mubold(dx) = \hbold (x) \sigma(dx) +  \mubold_s(dz)\,,\]
where  $\hbold$ is Hermitian nonnegative and $\mubold_s$ is singular with respect to $\sigma$, i.e. nonzero only on a set of $\sigma$ measure zero
Then, we define  the notion of Kullback-Leibler divergence
\begin{equation}
\label{kullback}
\mathcal K(\Sigma\!\  |\!\  \mubold) := - \int_{E} \log\det \hbold (x)\ \sigma (dx)\,.
\end{equation}
if $\log \det \hbold \in L^1(\sigma)$ and $\infty$ otherwise. 
We remark that it is possible to rewrite the above quantity in the flavour of Kullback-Leibler information (or relative entropy) with the notation of \cite{mandrekar} or \cite{robertson}.

Finally, we define $\boldsymbol{\Sr} = \boldsymbol{\Sr}(\am,\ap)$ the set of all bounded matricial measures $\mubold$ of size $r \times r$ such that
\begin{itemize}
\item[(i)] $\operatorname{supp}(\mubold) = J \cup \{E_i^-\}_{i=1}^{N^-} \cup \{E_i^+\}_{i=1}^{N^+}$, where $J\subset I= [\am,\ap]$, $N^-,N^+\in\N\cup\{\infty\}$ and 
\begin{align*}
E_1^-< E_2^-<\dots <\am \quad \text{and} \quad E_1^+> E_2^+>\dots >\ap .
\end{align*}
\item[(ii)] If $N^-$ (resp. $N^+$) is infinite, then $E_j^-$ converges towards $\am$ (resp. $\lambda_j^+$ converges to $\ap$).
\end{itemize}
Such a matricial measure $\mubold$ can always be written as
\begin{align}
\mubold = \mubold_{|I} +  \sum_{i=1}^{N^+} \boldsymbol{\Gamma}_i^+ \delta_{E_i^+} + \sum_{i=1}^{N^-} \boldsymbol{\Gamma}_i^- \delta_{E_i^-},
\end{align}
for some $r \times r$ matrices $\boldsymbol{\Gamma}_i^{\pm}$. We also introduce 
\[ \boldsymbol{\Sr}_1= \boldsymbol{\Sr}_1(\am,\ap):=\{\mu \in \boldsymbol{\Sr} |\, \mu(\R)= \mathbf{1} \}, \] 
and endow $\boldsymbol{\Sr}_1$ with the weak topology and the corresponding Borel $\sigma$-algebra.

In the unitary case, there is a corresponding framework. We omit to give details for simplicity. 
\bigskip
\paragraph*{{\bf The Hermitian case.}} For all $n \geq r$, let $X_n$ be a $\GUE(n)$ random matrix. 
 Let also $A_n$ be a deterministic Hermitian matrix having all of its entries equal to zero except for the $r \times r$ top-left block which is given by some Hermitian matrix $\mathbf{\Theta}$. We are interested in the matricial spectral measure of the deformed matrix:
\[ W_n := \frac{X_n}{\sqrt n} + A_n.   \]
The distribution $\mathbb P_{\mathbf{\Theta}}^{(n)}$ of $W_n$ is given by:
\[ \mathbb P_{\mathbf{\Theta}}^{(n)}(dW) = \frac{1}{\mathcal Z_n} \exp \left( -\frac{n}{2} \tr\!\ \big[ ( W - A_n)( W - A_n)^\star\big]  \right)  dW.  \]
Let $\mathbf{\mubold_\w}\sn$ be the matricial spectral measure associated to $W_n$ and the $r$-tuple $(e_1, \ldots,e_r)$. Since
\begin{align*}
\tr\!\ ( W - A_n)( W - A_n)^\star 
&= \tr\!\ (W W^\star) - 2 \tr\!\ (A_n W) + \tr\!\ (A_n A_n^\star) \\
&= \tr\!\ (W W^\star) - 2\!\ \tr ( \mathbf{\Theta} \mathbf{m}_1) + \tr\!\ \left( \boldsymbol{\Theta} \boldsymbol{\Theta}^\star \right),
\end{align*}
we deduce that
\begin{equation} \label{eq:SpecMesMatrHerm}
\mathbb P_{\mathbf{\Theta}}^{(n)}( \mathbf{\mubold_\w}\sn \in d \mathbf{\mubold}) = \frac{\exp n \Psi( \mubold) }{\mathbb{E}_0[ \exp n \Psi( \mathbf{\mubold_\w}\sn)]} \mathbb P_{\mathbf{0}}^{(n)}( \mathbf{\mubold_\w}\sn \in d \mathbf{\mubold}), 
\end{equation}
where $\mathbf{0}$ is the $r \times r$ matrix having all its coefficients equal to zero and where
\[  \Psi( \mubold ) =  \tr \left(   \mathbf{\Theta} \mathbf{m}_1( \mubold) \right).  \]
Under $ \mathbb P_{\mathbf{0}}^{(n)}$, it is known (see for example \cite{GaNaRomat}) that the sequence $( \mathbf{\mubold_\w}\sn)_{n \geq r}$ satisfies a large deviations principle at speed $n$ and with good rate function
\[ \mathcal{I}^X(\mathbf{\mubold}) =
\left\{
\begin{array}{lr} 
\mathcal{K}(SC \cdot \mathbf{1};\mathbf{\mubold}) + \sum\limits_{k \geq 1} \mathcal{F}_H(E_k^{\pm}) & \text{if } \mubold \in \boldsymbol{\Sr}_1(-2,2), \\
\infty & \text{otherwise}.
\end{array}
\right.  \]
Besides, note that for every $\gamma >0$, 
\[ \frac{1}{n} \log \mathbb{E}_{\mathbf{0}} \left[ \exp \tr \left(\mathbf{\Theta} \mathbf{m}_1  \right)   \right] = \frac{\gamma^2}{2} \tr \left( \boldsymbol{\Theta} \boldsymbol{\Theta}^\star \right).   \]
Therefore, applying Varadhan's Lemma to \eqref{eq:SpecMesMatrHerm}, we obtain the following analog of Theorem \ref{theo:perturbGUE}.

\begin{theorem}
\label{6.4}
The sequence $( \mathbf{\mubold_\w}\sn)_{n \geq r}$ satisfies a large deviations principle at speed $n$ and with good rate function $\mathcal{I}^W$ given by
\begin{align*}
\mathcal{I}^W(\mubold)= \begin{cases}
\mathcal{K}(SC \cdot \mathbf{1}\, |\ \mathbf{\mubold}) + \sum\limits_k \mathcal{F}_H(E_k^{\pm}) - \tr\!\ (  \mathbf{\Theta} \mathbf{m}_1  ) + \frac{1}{2} \tr\!\ ( \boldsymbol{\Theta} \boldsymbol{\Theta}^\star ) & \text{if } \mubold \in \boldsymbol{\Sr}_1(-2,2), \\
\infty & \text{otherwise}.
\end{cases}
\end{align*}
\end{theorem}
Let us finally describe the unique minimizer of $\mathcal{I}^W$. First, we claim that, as in the scalar case described in Section \ref{sec:recapOPRLOPUC}, there exists a one-to-one correspondence between matricial measures $\mubold$ and sequences of $r \times r$ matrices $(A_n)_{n \geq 1}$ and $(B_n)_{n \geq 1}$ such that the matrices $B_i$'s are Hermitian positive definite. Using the matricial sum rule (Th. 2.1 in  \cite{GaNaRomat}), the good rate function can be rewritten, when $\mubold \in \boldsymbol{\Sr}_1(-2,2)$:
\[ \mathcal{I}^W(\mubold) = \frac{1}{2} \tr\!\ \big[( B_1 - \boldsymbol{\Theta})( B_1 - \boldsymbol{\Theta}) ^\star\big]  + \frac{1}{2} \sum\limits_{n \geq 2} \tr\!\ (B_n B_n^\star) + \sum\limits_{n \geq 1} G(A_n A_n^\star).  \]
The unique minimizer $\mubold_{SC,\boldsymbol{\Theta}} = \operatorname{argmin} \mathcal I^W$ can therefore be described by its matricial Jacobi coefficients:
\[ \jac( \mubold_{SC,\boldsymbol{\Theta}} )= \begin{pmatrix}
\boldsymbol{\Theta}, & \mathbf{0},&  \mathbf{0},& \cdots\\
\mathbf{1}, & \mathbf{1},& \mathbf{1}, &\cdots
\end{pmatrix}. \]  
In order to obtain an explicit formula, we {use the matricial Stieltjes transform of $\mubold_{SC,\boldsymbol{\Theta}}$, defined by
\[  \mathbf{G}(z) := \int \frac{\mathrm{d}\mubold_{SC,\boldsymbol{\Theta}}(x)}{x - z \mathbf{1}}.  \]
By \cite[Theorem 4.3.3]{Simon-newbook}, it satisfies the following equation:
\begin{equation} \label{eq:StieltMatricialRel}
\mathbf{G}(z) = \left(\boldsymbol{\Theta} - \omega(z) \mathbf{1}\right)^{-1},
\end{equation}
where $\omega$ (called the subordination function) is here
\begin{align}
\omega(z) = z+G_{SC}(z),
\end{align}}
where $G_{SC}(z) = \frac{1}{2} ( z - \sqrt{z^2 - 4})$ is the Stieltjes transform of the semi-circle law. 

Since the absolutely continuous part of $\mubold_{SC,\boldsymbol{\Theta}}$ is given by
\[  \frac{d \mubold_{SC,\boldsymbol{\Theta}}(x)}{dx} = \lim\limits_{t \rightarrow 0^+} \frac{1}{\pi} \Im \mathbf{G}(x+it),  \]
it is easy to deduce that
\[ \frac{d \mubold_{SC,\boldsymbol{\Theta}}(x)}{dx} = \frac{\sqrt{(4-x^2)_+}}{ 2\pi} \left(\boldsymbol{\Theta}\boldsymbol{\Theta}^\star + \mathbf{1} - x \boldsymbol{\Theta}  \right)^{-1}.  \]
Moreover, $\mubold_{SC,\boldsymbol{\Theta}}$ has an atom at each pole of $\mathbf{G}$ and the mass of this atom is the corresponding residue. Thanks to \eqref{eq:StieltMatricialRel}, the poles of $\mathbf{G}$ corresponds to the reals $x$ such that 
\[\det (\boldsymbol{\Theta} -  \omega(x)\mathbf{1})=0\,.\] For simplicity, let us assume from now on that $\Theta$ has distinct eigenvalues $\theta_1, \ldots, \theta_r$, the adaptation in the general case being straightforward. Let $U$ be the matrix whose columns are the eigenvectors of $\boldsymbol{\Theta}$. Then, $\boldsymbol{\Theta} = U D U^\star$ with $D = \diag( \theta_1,\ldots, \theta_r)$, and we deduce that
\[  \mathbf{G}(z) = U \left(D - \omega(z)  \mathbf{1}\right)^{-1} U^\star.  \]
We now use the following well-known fact about the  function $\omega)$:
\begin{itemize}
\item if $|\theta| \leq 1$, there is no real $x$ such that $\omega(x) = \theta$;
\item if $|\theta|>1$, there exists exactly one real $x_\theta = \theta + \frac{1}{\theta}$ such that $|x_\theta|>2$ and $\omega(x_\theta) = \theta$. Moreover, $1/\omega'(x_\theta) = 1 - \frac{1}{\theta^2}$.
\end{itemize}
Therefore, the poles of $\mathbf{G}$ are in one-to-one correspondence with the eigenvalues $\theta_i$ of $\boldsymbol{\Theta}$ satisfying $|\theta_i|>1$, and each of this pole has a residue given by $U (1 - \frac{1}{\theta_i^2}) e_i e_i^T U^\star$. Hence, we have proved that:
\begin{align}\notag 
\mubold_{SC,\boldsymbol{\Theta}} (dx) &= \frac{\sqrt{(4-x^2)_+}}{ 2\pi} \left(\boldsymbol{\Theta}\boldsymbol{\Theta}^\star + \mathbf{1} - x \boldsymbol{\Theta}  \right)^{-1} dx \\ 
\label{muboldtheta}
&+  \sum\limits_{i=1}^r  \left( 1 - \frac{1}{\theta_i^2} \right) Ue_i e_i^T U^\star 1_{|\theta_i|>1} \delta_{ \theta_i + \frac{1}{\theta_i} }(dx). 
\end{align}
It can also be written as follows:
\[\mubold_{SC,\boldsymbol{\Theta}} 
 = U \diag(\mu_{\SC, \theta_1}, \cdots, \mu_{\SC, \theta_r})U^*\,.\]

\paragraph*{{\bf Application.}} In Section 5.1, we have considered the spectral measure $\mu_\w\sn$ of the pair $(W_n , e)$ for a rank-one perturbation $\theta uu^*$ when $u=e$. The matricial theory allows to consider the case $u\not=e$.  
 Assume $\langle u, e\rangle = \cos \varphi\not= \pm 1$ and consider the following orthonormal basis. 
We set  $e_1 = e$, $w = u -\langle u, e\rangle e$,  $e_2 = \frac{w}{\Vert w\Vert}$ and we complete by $e_3, \cdots, e_n$. 
We can now consider the random matrix
\[W_n = \frac{X_n}{\sqrt n} + A_n\,,\]
where $A_n$ is a  rank-one deterministic Hermitian matrix, having all its entries equal to zero except for the $2 \times 2$ top-left block matrix which is  $\Tbold = \theta \Rbold$ with
\[\Rbold 
= \begin{pmatrix}\cos^2 \varphi& \sin \varphi \cos \varphi\\ \sin \varphi \cos \varphi& \sin^2\varphi 
\end{pmatrix}\,.\]
$\Rbold$ is a projection and then
\begin{align}\notag
\Tbold\Tbold^* + \bdone -z\Tbold &= \bdone - (\theta z- \theta^2)\Rbold\\
\label{inverse}
(\Tbold\Tbold^* + \bdone -z\Tbold)^{-1} 
&= \bdone + \left((\theta^2 + 1 - \theta z)^{-1} - 1\right)\Rbold\,,
\end{align}
as soon as $\theta^2 + 1 - \theta z \not=0$. Moreover
\[U= \begin{pmatrix}\cos \varphi &- \sin\varphi\\
\sin\varphi & \cos\varphi
\end{pmatrix}\,.\]
The (scalar) spectral measure $\mu\sn$ of the pair $(W_n, e)$  is exactly $\left(\mubold_\w\sn\right)_{11}$. The equilibrium measure 
is $\left(\mubold_{\SC, \Tbold}\right)_{11}$ and then, from (\ref{muboldtheta})  and (\ref{inverse})  
\begin{align}
  \left(\mubold_{\SC, \Tbold}\right)_{11} =  (\sin^2\varphi) \SC +  (\cos^2 \varphi) \mu_{\SC, \theta}\,.\end{align}
i 
As seen above,  $\mubold_\w\sn$ satisfies the LDP in the scale $n$ and then, 
 by the contraction principle, $\mu\sn$  satisfies the LDP  with rate function
\[\mu\in \mathcal M_1(\mathbb R) \mapsto  \inf \{\mathcal I^W(\mubold) ; \mubold \in \mathcal M_{2,1}(\mathbb R) , (\mubold)_{11} = \mu\}\,.\]
where $\mathcal I^W$ was defined in Theorem \ref{6.4}
but we didn't find an expression of this rate function.
\medskip

\paragraph*{{\bf The Gross-Witten case.}} The role of $e^{\ii \varphi}$ is now played by a unitary $r\times r$ operator.
In the sequel, we will omit the subscript $n$ to simpifly the notation. 
 As in  (4.5.10) in \cite{simon05} we consider 
\begin{align}
W = UQ \ , \  Q= 1+ (\Lambda - 1) P\,,
\end{align}
where $P$  is the projection on $\mathcal H_r = \hbox{Vect}\ \{ e_1, \dots, e_r\}$
 and $\Lambda$ is a unitary operator acting on $\mathcal H_r $. 
Notice that
\[Q^{-1} = Q^* = 1+ (\Lambda^* -1) P\,.\]
In other words,
\[Q = \Lambda\oplus I_{n-r}\ , \ Q^{-1} = Q^* =  \Lambda^*\oplus I_{n-r}\,.\]
If $\mu$ is the spectral measure of the pair $(U ; e_1, \dots, e_r)$, let us denote by $\tau_\Lambda\mu$ the spectral measure of the pair $(W; e_1, \dots, e_r)$. 
We have the matricial version of (\ref{alex0})
(Theorem 4.5.6 in \cite{simon05})
\begin{align}
\notag
F_\Lambda &= \left[(\bdone + \Lambda) - F (\bdone-\Lambda)\right]^{-1}\left[-(\bdone - \Lambda) + F (\bdone + \Lambda)\right]
\,.
\end{align}
which gives, via the Schur recursion
\begin{align}
\label{changealphaM}
\abold_k(\tau_\Lambda\mu) = \Lambda^*\abold_k(\mu) \ , \ (k \geq 0)\,.
\end{align}

To compute 
the distribution of $W$, let us denote by $W^\uparrow_r$ the $r\times r$ upper left corner of $W$ and by $W^\downarrow_{n-r}$  the $(n-r)\times (n-r)$ lower right corner of $W$ so that
\begin{align}
\label{6.17}
\mathbb P_\Lambda\sn (dW) = \frac{1}{\mathcal Z_0\sn}\exp \frac{n\g}{2} \tr\! \left (WQ^{-1}+ (WQ^{-1})^*\right) dW\,.
\end{align}
Since
\begin{align}\notag
\tr (WQ^{-1}) & 
= \tr\!\ W+ \tr\!\ \left(W_r^\uparrow (\Lambda^* - \bdone)\right)\\
\tr\left(WQ^{-1} + (WQ^{-1})^*\right)&= \tr\!\ ( W + W^*) + 2 \Re\tr\!\ \left(W_r^\uparrow (\Lambda^* - \bdone)\right)\,,
\end{align}
(\ref{6.17}) may be written
\begin{align}
\mathbb P_{\Lambda}\sn (dW) =\exp n\g \Re\tr\!\ \left(W_r^\uparrow (\Lambda^* - \bdone)\right)\ 
\mathbb P\sn (dW) \,.
\end{align}
Let $\mubold_\w\sn$ be the matricial spectral measure of $(W ; e_1, \dots, e_r)$.  Since $W_r^\uparrow = \abold_0^* = \mbold_1(\mu_\w\sn)$, we get
\begin{align}
\mathbb P_{\Lambda}\sn (\mu_\w\sn \in d\mu) = \exp n\g \Re\tr\!\ \left(\mbold_1(\mu)(\Lambda^*  - \bdone)\right)\ 
\mathbb P_{\bdone}\sn (\mu_\w\sn \in d\mu)\,. 
\end{align}
Under $\mathbb P_{\bf 1}\sn$, it is known {(\cite{GNROPUC}) that the sequence $(\mubold _\w\sn)_{n \geq r}$, satisfies an LDP  at speed $n$.
 If $|\g| \leq 1$, the rate function is
\begin{align}
\mathcal I^{\GW} (\mubold) = \mathcal K(\GW_{\g}\cdot \bdone\, |\ \mubold)\,.
\end{align}
The matrix measure $\GW_\g \cdot \bdone$ is the unique minimum of $\mathcal I^{\GW}$.
 
 This allows to 
 obtain the following analog of Theorem \ref{corGW}.
\begin{theorem}
\label{6.5}The sequence $(\mubold_\w\sn)_{n \geq r}$ satisfies an LDP at speed $n$ and good rate function
\begin{align}
\lb{rateGWM}
\mathcal I^W (\mubold) = \mathcal I^{GW}(\mubold) 
 -\g \Re\tr\!\ \left(\mbold_1(\mu) (\Lambda^* - \bdone)\right) \,.
\end{align}
\end{theorem}
There is a matrix version of the method to recover the measure (Prop. 3.16 in \cite{damanik2008analytic} and Lemma 7.1 in \cite{bolotnikov2006boundary}). 
From (\ref{changealphaM}), it is then straightforward to state that if $d\mubold(\theta) = w(\theta)\cdot\bdone \ d\lambda_0(\theta) + d\mubold_s(\theta)$, then $\tau_\Lambda(\mubold\cdot \bdone)$ has for density 
\begin{align}
w^\Lambda(\theta) = 4 w(\theta) \left|\bdone + \Lambda + F(\theta) (\bdone - \Lambda)\right|^{-2}
\end{align}
where $|A|^2 = AA^*$ 
(analog of (\ref{5.29}). Notice that if $|\g|\leq 1$  there is no extra mass.

\noindent From (\ref{changealphaM}) we have
\begin{align}
\mathbb P_\Lambda\sn (\mubold_\w\sn \in d\mubold) = \mathbb P_\bdone\sn \left(\mubold_\w\sn \in d(\tau_{\Lambda^*} \mubold)\right)\,.
\end{align}
Under $\mathbb G\mathbb W_{\g}\sn$, the rate function for the LDP  is $\mathcal K(\GW_{\g}\cdot \bdone \, |\ \mubold)$.
A pushforward of this LDP gives 
\begin{align}
\label{5.25M}
\mathcal I^W(\mubold) = \mathcal K(\GW_{\g}\cdot \bdone\, |\ \tau_\Lambda\mubold)\,.
\end{align}
It is then clear that $\mathcal I^W$ reaches his unique minimum at $\tau_{\Lambda*}(\GW_\g \cdot \bdone)$.

\begin{remark}

We don't give an alternatative proof of the LDP. Actually we could have used 
the matrix version of the sum rule
 (\ref{sumrulegwg}) proved recently by analytic methods in \cite{ARH}: 
 \begin{align}
\label{mainfou}
\mathcal K(\GW_{-\g}\cdot \bdone \, | \  \mubold) = r \mathcal K(\GW_{-\g}\cdot \bdone \, | \   \lambda_0) 
+ \Re \tr (\abold_0) +\frac{\g}{2}\tr \abold_0\abold_0^\dagger + \sum_0^\infty T_\g (\abold_k)\\ + \frac{\g}{2}\sum_0^\infty \tr (\abold_{k+1} - \abold_k)(\abold_{k+1} - \abold_k)^\dagger
\end{align}
where
\[T_\g (\abold) = - \log\det (\bdone - \abold\abold^\dagger) -\g \tr\abold\abold^\dagger
\,.\]
 Replacing $\abold_k$ by $\abold_k e^{i \varphi}$ allows to recover :
\begin{align}
\mathcal K(\GW_{-\g}\cdot \bdone \, | \  \tau_{e^{-i\varphi}}\mubold) = \mathcal K(\GW_{-\g}\cdot \bdone \, | \  \mubold) + \Re \tr\!\ \left(\abold_0( e^{i \varphi}-1\right)\,.
\end{align}
\end{remark}

\section{Appendix}
We use the affine transformation $T_{\alpha, \beta}$ corresponding to the change of variable $x = \alpha y + \beta$. 

\subsection{A technical result}
The  first lemma is elementary. We give its proof for the sake of completeness.
\begin{lemma}
\label{affine}
If
\begin{align}\jac(\mu) 
= \begin{pmatrix}
b_1, & b_2,& \cdots\\
a_1, &a_2, &\cdots
\end{pmatrix}
\end{align}
then
\begin{align}\jac\left(T_{r,s}(\mu)\right) 
= \begin{pmatrix}
\tilde b_1, & \tilde b_2,& \cdots\\
\tilde a_1, &\tilde a_2, &\cdots
\end{pmatrix}\ \hbox{with} \ \ \tilde a_k = \frac{a_k}{|r|} \ , \ \tilde b_k = \frac{b_k -s}{r}\,.
\end{align}
\end{lemma}

\proof If $J$ be the Jacobi matrix associated with $\mu$
\[\langle e , (J-z)^{-1}e\rangle = \int \frac{d\mu(x)}{x-z}\]
hence
\begin{align*}\int \frac{dT_{r,s}(\mu)(y)}{y-z}= \int \frac{d\mu(x)}{r^{-1}(x-s)-z} 
= \langle e ,  (r^{-1}(J-s) -z)^{-1}\rangle \end{align*}
hence if $r >0$ the Jacobi matrix associated to $T_{r,s}(\mu)$ is $\tilde J = r^{-1}(J-s)$. 

If $r =-1, s=0$, the tridiagonal operator $-J$ admits $T_{-1,0}$ as its spectral measure, but $-J$ is not Jacobi. A change of basis $e_k \mapsto \tilde e_k = (-1)^{k-1}e_k$ gives the true Jacobi with $\tilde b_k = \langle \tilde e_k,  (-J)\tilde e_k\rangle = -b_k$ and $\tilde a_k = \langle \tilde e_{k+1}, (-J)\tilde e_k\rangle = a_k$.

\subsection{Free Meixner distributions}
\label{freeM}
From \cite{anshelevich2011bochner},  we know\footnote{Be careful, the author  considered the sequence $\{a_n^2 , b_n\}$ as Jacobi coefficients.} that the normalized free Meixner distributions $\mu_{b,c}$  are probability
measures on $\mathbb R$ with Jacobi parameter sequences
\begin{align}
\label{JacFreeM}\jac(\mu_{b,c}) 
= \begin{pmatrix}
0, & b,& b,& \cdots\\
1, & \sqrt{1+c},& \sqrt{1+c}, &\cdots
\end{pmatrix}
\end{align}
$b \in \mathbb R, c > -1$.
The first line corresponds to the $b$'s (diagonal terms) and the second to the $a$'s (subdiagonal terms).
The corresponding probability measure is 
\begin{align}\label{gen}\mu_{b,c}(dx) := \frac{1}{2\pi}\cdot \frac{\sqrt{4(1+c) - (x-b)^2}}{1+bx + cx^2} dx + p_1 \delta_{x_1} + p_2\delta_{x_2}\,,\end{align}
where $x_1$ and $x_2$ are real roots of $1+bx+cx^2 = 0$ (if there exist(s)) and $p_1, p_2 \in [0,1)$.
The mean is $0$ and the variance is $1$. 

The case $b=c=0$ and $p_1=p_2=0$ is just SC also called {\bf "free Gaussian"}.

\noindent In order to compare $\mu_{b,c}$ with  SC, we transform the support into $[-2,2]$ and set
\begin{align}
\label{geny}
\tilde\mu_{b,c} (dy) :=T_{\sqrt{1+c}, b}(\mu_{b,c})(dy) := \frac{1}{2\pi}\cdot \frac{\sqrt{4-y^2}}{cy^2 + \alpha y + \beta}dy + p_1 \delta_{y_1} + p_2\delta_{y_2}
\end{align}
with 
\begin{align}
\label{117}
\jac\left(\tilde\mu_{b,c}
\right)
 = \begin{pmatrix} -b/\sqrt{1+c},&0,&0,&\cdots\\
1/\sqrt{1+c},&1,&1,&\cdots
\end{pmatrix}\,.
\end{align}

Apart from SC  there are only 5 situations.

\begin{enumerate}
\item $c=0$, 
 ($b\not= 0$).
\begin{align}\mu_{b,0}(dx) &=  \frac{1}{2\pi}\cdot \frac{\sqrt{4 - (x-b)^2}}{1+bx } + (1 - b^{-2})^+ \delta_{-b^{-1}}\\
\label{MP}
T_{1,b}(\mu_{b, 0}) (dy) &=\frac{1}{2\pi} \frac{\sqrt{4-y^2}}{(1+b^2) + by} dy + (1 - b^{-2})^+\!\  \delta_{-b -b^{-1}}\,.\end{align}
It is a variant of $\MP$, called also {\bf "free Poisson"}. Indeed, 
\[T_{b,1}(\mu_{b,0})(dy) =  \frac{1}{2\pi b^2}\cdot \frac{\sqrt{((1+b)^2 -y)(y - (1-b)^2)}}{y} dy + (1 - b^{-2})^+ \delta_0\,\]

\item $c\not = 0$
\begin{enumerate}
\item $-1 < c <0$,   it is called {\bf "free binomial"}, the denominator has two real roots. For instance, when $b=0$ we get the measure
\begin{align}
\mu_{0,c} (dx) =  \frac{1}{2\pi}\cdot \frac{\sqrt{4(1+c) -x^2}}{1 + cx^2} dx + p \left(\delta_{-1/\sqrt{(-c)}} +  \delta_{1/\sqrt{(-c)}} \right)\,,
\end{align}
with $ p =\left(1 + \frac{1}{2c}\right)^+$, 
\begin{align}\notag T_{\sqrt{1+c}, 0}(\mu_{0,c})(dy)&= \frac{1}{2\pi}\cdot \frac{\sqrt{4 -y^2}}{(1 +c)^{-1} + cy^2} dy \\&+ p\left(\delta_{-1/\sqrt{-c(1+c)}} + \delta_{1/\sqrt{-c(1+c)}}\right)\,. \end{align}

Notice that  the variance is $\sigma^2 = 1/(1+c) > 1$. There are  masses if and only if $c \in (-1, -1/2)$. 

Up to an affine transform, this distribution is of the KMK type. 
 In other words it is the equilibrium measure when the potential is $-n \kappa_2 \log x -n \kappa_1 \log (1-x)$ (see Appendix)
\item $c >0, b^2 -4c <0$ , for instance with $b=0$. We get
\begin{align}
\mu_{0,c} (dx) = \frac{1}{2\pi}\cdot \frac{\sqrt{4(1+c) -x^2}}{1 + cx^2} dx  
\end{align}
(without any atoms). 
It is called {\bf"free hyperbolic tangent"}  or {\bf"free Meixner type"}, and
\[T_{\sqrt{1+c}, 0}(\mu_{0,c})(dy) = \frac{1}{2\pi}\cdot \frac{\sqrt{4 -y^2}}{(1 +c)^{-1} + cy^2} dy\,.\]
Notice that the variance is $\sigma^2 = 1/(1+c) <1$.
Up to a scaling, this distribution can be obtained  by Cayley transform from the Hua-Pickrell distribution. In other words it is the equilibrium measure when the potential is  $n \log (1+x^2)$ (see \cite{GNROPUC}). 

\item  $b^2 = 4c$, one double root $x = -2/b$, the measure is
\[\mu_{b, b^2/4}(dx) = \frac{1}{2\pi}\cdot \frac{\sqrt{4 + 2bx -x^2}}{\left(1 + \frac{bx}{2}\right)^2} dx\,.\]
It is sometimes called {\bf "free Gamma type"} and

\begin{align}
\label{freeG}
T_{\sqrt{1+ \frac{b^2}{4}}, b}(\mu_{b, b^2/4})(dy) = \frac{1}{2\pi}\cdot \frac{\sqrt{4-y^2}}{\left(\frac{b}{2}y + \frac{b^2+2}{\sqrt{b^2+4}}\right)^2} dy\,. 
\end{align}
\item $c > 0, b^2-4c >0$, it is called {\bf "free Pascal"}, the denominator in (\ref{gen}) has two real roots
\[x_\pm = -\frac{b}{2c} \pm \sign b \frac{\sqrt{b^2-2c}}{2c}\]
and there is a mass $p = \left(1 - \frac{|b| - \sqrt{b^2-4c}}{2c\sqrt{b^2-4c}}\right)^+$ at $x_+$, and
\begin{align}T_{\sqrt{1+c},b}(\mu_{b,c}) (dx) =   \frac{1}{2\pi}\cdot \frac{\sqrt{4-y^2}}{c(y-y_+)(y-y_-)} + p\delta_{y_+}\,,\end{align}
where $y_+ = \frac{x_+ -b}{\sqrt{1+c}}$. 
\end{enumerate}
\end{enumerate}

\acknowledgement{A.R. thanks Fabrice Gamboa and Jan Nagel for valuable comments on this work.}

\bibliographystyle{plain}

\begin{thebibliography}{10}

\bibitem{agz}
G.~Anderson, A.~Guionnet and O.~Zeitouni.
\newblock {\em An introduction to random matrices}.
\newblock Cambridge University Press, Cambridge, 2010.

\bibitem{anshelevich2011bochner}
M.~Anshelevich.
\newblock Bochner--{P}earson-type characterization of the free {M}eixner class.
\newblock {\em Adv. Appl. Math.}, 46(1-4):25--45, 2011.

\bibitem{BBPphase}
J.~Baik, and G.~Ben Arous and S.~P\'ech\'e.
\newblock Phase transition of the largest eigenvalue for nonnull complex sample covariance matrices
\newblock {\em Ann.  Probab.}, 1643--1697, 2005.

\bibitem{DJ}
J.~Baik, P.~Deift and K.~Johansson.
\newblock On the distribution of the length of the longest increasing
  subsequence of random permutations.
\newblock {\em J. Amer. Math. Soc.}, 12(4):1119--1178, 1999.

\bibitem{BGGM}
F.~Benaych-Georges, A.~Guionnet and M.~Maida.
\newblock Large deviations of the extreme eigenvalues of random deformations of
  matrices.
\newblock {\em Probab. Th. Rel. Fileds}, 154:703--751, 2012.

\bibitem{biroli}
G.~Biroli and A.~Guionnet.
\newblock Large deviations for the largest eigenvalues and eigenvectors of
  spiked random matrices.
\newblock {\em Electron. Commun. Probab.
},25,  2020.

\bibitem{bolotnikov2006boundary}
V.~Bolotnikov and H.~Dym.
\newblock On boundary interpolation for matrix valued {S}chur functions.
\newblock {\em Mem. Amer. Math. Soc.}, 181(856):vi+107, 2006.

\bibitem{borotgui2013}
G.~Borot and A.~Guionnet.
\newblock Asymptotic expansion of $\beta$ matrix models in the one-cut regime.
\newblock {\em Comm. Math. Phys.}, 317(2):447--483, 2013.

\bibitem{BSZ}
J.~Breuer, B.~Simon and O.~Zeitouni.
\newblock Large deviations and the {L}ukic conjecture.
\newblock {\em Duke Math. J.}, 167(15):2857--2902, 2018.

\bibitem{Bryc}
W.~Bryc.
\newblock Free exponential families as kernel families.
\newblock{\em Demostratio Math.} XLII (3): 657-672, 2009.

\bibitem{DMC}
M.~Capitaine and C.~Donati-Martin.
\newblock Spectrum of deformed random matrices and free probability.
\newblock In {\em Advances topics in random matrices}. SMF, 2017.

\bibitem{damanik2008analytic}
D.~Damanik, A.~Pushnitski and B.~Simon.
\newblock {The analytic theory of matrix orthogonal polynomials}.
\newblock {\em Surv. Approx.Theory}, 4:1--85, 2008.

\bibitem{demboz98}
A.~Dembo and O.~Zeitouni.
\newblock {\em Large Deviations Techniques and Applications}.
\newblock Springer, 1998.

\bibitem{deuschel2001large}
J-D. Deuschel and D.~Stroock.
\newblock {\em Large deviations}, volume 342.
\newblock American Mathematical Soc., 2001.

\bibitem{ellis1985entropy}
R.~Ellis.
\newblock{\em Entropy, large deviations, and statistical mechanics}
\newblock Springer, 1985.

\bibitem{GaNaRo}
F.~Gamboa, J.~Nagel and A.~Rouault.
\newblock Sum rules via large deviations.
\newblock {\em J. Funct. Anal.}, (270):509--559, 2016.

\bibitem{GNROPUC}
F.~Gamboa, J.~Nagel and A.~Rouault.
\newblock Sum rules and large deviations for spectral measures on the unit
  circle.
\newblock {\em Random Matrices Theory Appl.}, 6(1):1750005, 49, 2017.

\bibitem{GaNaRomat}
F.~Gamboa, J.~Nagel and A.~Rouault.
\newblock Sum rules and large deviations for spectral matrix measures.
\newblock {\em Bernoulli}, 25(1):712--741, 2018.

\bibitem{GNRis}
F.~Gamboa, J.~Nagel and A.~Rouault.
\newblock Sum rules via large deviations: extension to polynomial potentials
  and the multi-cut regime.
\newblock {\em To appear in J. Funct. Anal., preprint arXiv:2004.13566}, 2020.

\bibitem{Gateway}
F.~Gamboa, J.~Nagel and A.~Rouault.
\newblock Some gateways between some sum rules.
\newblock{\em In preparation}

\bibitem{gamboacanonical}
F.~Gamboa and A.~Rouault.
\newblock {Canonical moments and random spectral measures}.
\newblock {\em J. Theoret. Probab.}, 23:1015--1038, 2010.
\newblock Erratum in the same journal (2015) doi 10.1007/s10959-015-0653-5.

\bibitem{gamboa2011large}
F.~Gamboa and A.~Rouault.
\newblock Large deviations for random spectral measures and sum rules.
\newblock {\em Applied Mathematics Research eXpress}, 2011(2):281--307, 2011.

\bibitem{GrossW}
D.J. Gross and E.~Witten.
\newblock {Possible third-order phase transition in the large-N lattice gauge
  theory}.
\newblock {\em Phys. Rev. D}, 21(2):446--453, 1980.

\bibitem{HiaiP}
F.~Hiai and D.~Petz.
\newblock {\em The Semicircle Law, Free Random Variables and Entropy},
  volume~77 of {\em Mathematical Surveys and Monographs}.
\newblock Amer. Math. Soc., Providence, 2000.

\bibitem{johansson1998fluctuations}
K.~Johansson.
\newblock On fluctuations of eigenvalues of random {H}ermitian matrices.
\newblock {\em Duke Math. J.}, 91(1):151--204, 1998.

\bibitem{Johnstone}
I.~Johnstone.
\newblock On the distribution of the largest eigenvalue in principal components analysis
\newblock {\em Ann. Statis.}, 295--327, 2001.

\bibitem{killip2003sum}
R.~Killip and B.~Simon.
\newblock Sum rules for {J}acobi matrices and their applications to spectral
  theory.
\newblock {\em Ann. of Math. (2)}, 158(1):253--321, 2003.

\bibitem{knowles-yin}
A.~Knowles and J.~Yin.
\newblock Anisotropic local laws for random matrices.
\newblock {\em Probab. Theory Rel. Fields}, (169):257--362, 2017.

\bibitem{Kozhan}
R.~Kozhan.
\newblock Finite range perturbations of finite gap Jacobi and CMV operators.
\newblock{\em Adv. Math.}, (301): 204-226, 2016.

\bibitem{lee2015edge}
J.O.~Lee and K.~Schnelli.
\newblock Edge universality for deformed {W}igner matrices.
\newblock {\em Rev. Math. Phys.}, 27(08):1550018, 2015.

\bibitem{Lenczewski}
R.~Lenczewski.
\newblock Random matrix model for free {M}eixner laws.
\newblock {\em Int. Math. Res. Not. IMRN}, (11):3499--3524, 2015.

\bibitem{mandrekar}
V.~Mandrekar and H.~Salehi.
\newblock On singularity and Lebesgue type decomposition for operator-valued measures.
\newblock{\em J. Multivariate Anal.}, 1(2) : 167--185, 1971.

\bibitem{noiry2019spectral}
N.~Noiry.
\newblock Spectral measures of spiked random matrices.
\newblock {\em J. Theoret. Probab.}, 34(2): 923-952, 2021.

\bibitem{pastur2011eigenvalue}
L.~Pastur and M.~Shcherbina.
\newblock {\em Eigenvalue distribution of large random matrices}.
\newblock Number 171. American Mathematical Soc., 2011.

\bibitem{robertson}
J.B.~Robertson and M.~Rosenberg.
\newblock The decomposition of matrix-valued measures.
\newblock{\em Michigan Math. J.}, 15: 353-368, 1968.

\bibitem{ARH}
A.~Rouault.
\newblock A matrix version of a higher-order {S}zeg{\H o} theorem.
\newblock{\em J. Approx. Th.}, 266, 2021.

\bibitem{simon05}
B.~Simon.
\newblock {\em {Orthogonal polynomials on the unit circle. Part 1: Classical
  theory.}}
\newblock {Colloquium Publications. American Mathematical Society 54, Part 1.
  Providence, RI: American Mathematical Society (AMS)}, 2005.

\bibitem{Simon2}
B.~Simon.
\newblock {\em {Orthogonal polynomials on the unit circle. Part 2: Spectral
  theory.}}
\newblock {Colloquium Publications. American Mathematical Society 51, Part 2.
  Providence, RI: American Mathematical Society}, 2005.

\bibitem{simon-rk1}
B.~Simon.
\newblock Rank one perturbations and the zeros of para-orthogonal polynomials
  on the unit circle.
\newblock {\em J. Math. Anal. Appl.}, 329(1):376--382, 2007.

\bibitem{Simon-newbook}
B.~Simon.
\newblock {\em Szeg{\H o}'s theorem and its descendants}.
\newblock M. B. Porter Lectures. Princeton University Press, Princeton, NJ,
  2011.

\bibitem{Wadia}
S.~Wadia.
\newblock A study of {U(N)} lattice gauge theory in 2-dimensions.
\newblock {\em arXiv preprint arXiv:1212.2906}, 2012.

\bibitem{Webb}
M.~Webb and S.~Olver
\newblock Spectra of Jacobi operators via connection coefficient matrices
\newblock{\em Comm. Math. Phys.}, 382(2):687-707, 2021.

\bibitem{XYY}
H.~Xi, F.~Yang and J.~Yin.
\newblock Convergence of eigenvector empirical spectral distribution of sample
  covariance matrices.
\newblock {\em Ann. Statist.}, 48(2):953--982, 2020.




\end{thebibliography}

\end{document}